%% file: main_arxiv.tex
\documentclass[a4paper,11pt, reqno]{amsart}
\usepackage[margin=2cm]{geometry}

\usepackage[authoryear]{natbib}

\usepackage{amsmath}
\usepackage{amssymb}
\usepackage{amsthm}
\usepackage{color}
\usepackage{graphicx}
\usepackage[colorlinks=true, allcolors=blue]{hyperref}
\usepackage{setspace}
\usepackage{subcaption}
\usepackage{rotating}
\usepackage{booktabs}
\usepackage{multirow}
\def\T{\mathcal{T}}

\def\tG{\widetilde{G}}
\def\tV{\widetilde{V}}
\def\tA{\widetilde{A}}
\def\tP{\widetilde{P}}

\definecolor{atomictangerine}{rgb}{1.0, 0.6, 0.4}

\newtheorem{thm}{Theorem}[section]

\newtheorem{prop}{Proposition}[section]

\newtheorem{ex}{Example}[section]

\usepackage{multirow}

\usepackage{tikz}
\usetikzlibrary{arrows.meta, positioning, fit}
\newcommand*\circled[1]{\tikz[baseline=(char.base)]{
            \node[shape=circle,draw,inner sep=1.25pt] (char) {\small #1};}}
\newcommand*\squared[1]{\tikz[baseline=(char.base)]{
            \node[shape=rectangle,fill=gray!20,draw,inner sep=2pt] (char) {#1};}}

\let\origmaketitle\maketitle
\def\maketitle{
	\begingroup
	\def\uppercasenonmath##1{} 
	\let\MakeUppercase\relax 
	\origmaketitle
	\endgroup
}

\begin{document}

\title[]{\Large Coordinating Drop-Off Locations and Pickup Routes:\\
A Budget-Constrained Routing Perspective}
\author[M. Albareda-Sambola, V. Blanco, \MakeLowercase{and} Y. Hinojosa]{
{\large Maria Albareda-Sambola$^{\dagger}$, V\'ictor Blanco$^{\ddagger}$, and Yolanda Hinojosa$^{\star}$}\medskip\\
$^\dagger$Universitat Politècnica de Catalunya, Spain\\
$^\dagger$Institute of Mathematics (IMAG), Universidad de Granada, Spain\\
$^\star$Institute of Mathematics (IMUS), Universidad de Sevilla, Spain\\
\texttt{maria.albareda@ups.edu},
\texttt{vblanco@ugr.es}, \texttt{yhinojos@us.es}
}

\maketitle

\begin{abstract}
We introduce in this paper a new variant of a location routing problem, to decide,   the number and location of drop-off points to install based on the demands of a set of pick-up points, according to a given set-up budget for installing drop-off points.  A single vehicle is in charge for all pick-up and drop-off operations, and the solution cost is associated with its route, which must also be decided.

We provide a general and flexible mathematical optimization based approach for solving the problem that has some peculiarities to assure that the demand is adequately picked up, that some pickup points can be visited multiple times, that the capacity of the vehicle is respected, or that the vehicle is capable to implement the path or tour in the obtained solution. We report the results of a extense battery of experiments to validate our proposal on synthetic instances, and provide some insighs on the usefulness of our approach in practical applications.
\end{abstract}

\keywords{
Pickup and Drop-off, Facility Location, Vehicle Routing Problem}

\section*{Introduction}

The efficient distribution and  collection of goods in a supply chain and the optimal placement of facilities or drop-off points are crucial, particularly for last-mile operations. Last-mile delivery often involves picking up products from various facilities and ensuring timely, accurate deliveries to final destinations, which can be dispersed and unpredictable. Poorly planned distribution networks can lead to increased costs, delays, and inefficiencies, impacting customer satisfaction and business performance. By strategically positioning drop-off points and streamlining the routing of goods between pick-up locations and final delivery spots, companies can reduce transportation costs, minimize delivery times, and enhance the flexibility needed to handle varying demand. Optimal placement of these points also supports sustainability efforts by reducing unnecessary trips and fuel consumption, contributing to both economic and environmental benefits.

This type of supply chain is particularly important in the case study that  initially motivated this research: the design of the distribution network to produce biogas. Different organizations and governments have already signed agreements to ensure affordable, reliable, sustainable, and modern energy by 2030 to 2050~\citep{ONU,IEA,IPCC,EGD}, with a battery of measures designed to mitigate the effects of climate change. One of the main strategies to achieve the proposed goal is the use of biogas as a sustainable alternative to carbon-based energies, since it contributes to the reduction of greenhouse gases and the development of the circular economy through the anaerobic digestion of organic waste from different sources, and its transformation into fuel. The production of biogas from manure or other types of organic waste requires the implementation of a complex supply chain network~\cite[see, e.g.][]{blanco2024waste} to collect the waste from the farms, pre-process it in specialized pre-treatment plants (where the product is dried and cleaned), deliver the result to the anaerobic digestion plants to produce the biogas, and then distribute the final product (in the form of biogas or digestate) to different types of users (to an existing gas network, to external customers by liquefying the gas and shipping it, or back to the farms). When designing the network from scratch, both the location of the different types of plants and the distribution system have to be decided. Thus, it is necessary to develop technological tools to help the different agents involved in the process. At this point, mathematical optimization has been proven to be a useful framework to construct robust and effective supply chains with different characteristics \citep[see e.g.][and the references therein]{DEMEYER2015247, sarker_optimal_2018, TOMINAC2022107666, CAMBERO201462, GHADERI2016972}. Specifically, some successful approaches have been proposed to derive aid-decision tools for the management and design of supply chains for the production of biogas.  \cite{jensen_optimizing_2017}, propose a minimum cost flow-based model to assure economically feasible biogas plant projects for an existing supply chain network. In \cite{sarker_modeling_2019}, the biomethane production system is cost-optimized assuming that the process consists of four phases: collecting feedstock to hubs located according to zip code areas, transporting feedstock from hubs to reactor(s), transporting biomethane from reactor(s) to condenser(s), and shipping the liquefied biomethane from condensers to final demand points. More recently, \cite{aal2024matheuristic} provides a mathematical optimization model and a metaheuristic algorithm to maximize the profit of a biomass supply chain designed to commercialize electricity. The complete process of the so-called \emph{Waste to Biomethane Logistic Problem} is given in \cite{blanco2024waste}, where all the elements involved in the supply chain (from the collection of the manure/waste from farms to the provision of biogas to the end users) are integrated into a single mathematical optimization model.

In this paper, we focus on the first phase of the whole biogas supply chain system, which consists of collecting the waste and unload it to the so-called pre-treatment plants, where the manure or organic waste is prepared (dried, cleaned, etc.) to be delivered in the next phase to the anaerobic digestion plants to be transformed into biogas. In this phase, decisions on both, the locations of the pre-treatment plants and the design of the pickup routes  need to be made.
In previous approaches \citep[][]{blanco2024waste,jensen_optimizing_2017,sarker_modeling_2019}, it is assumed that a vehicle collects the waste from the farms and delivers it directly to their closest plant. This assumption can be restrictive and costly, and may result in inefficiencies when the number of farms is large, their production is significantly smaller than a truckload, and the budget for installing plants is limited. We propose here an alternative collection/drop-off policy based on two principles: (1) the vehicle can visit different pickup points to collect the product in the same route; and (2) the plants (drop-off points) can be located in intermediate positions in the vehicle's route, allowing the vehicle to unload the collected product and continue the route to other pickup points and plants. Thus, we provide a solution method based on mathematical optimization models to decide where to locate the plants that the agent can afford with a given budget, and how to design the visits to the pickup points and the plants to minimize the overall transportation costs. More specifically, we are given a finite set of pickup points, each endowed with a demand, a finite set of potential locations for the plants, and a capacity for the (single) vehicle that will pick up and deliver the product. We are also given a maximum budget to install the plants, and cost (distance/time) matrices to transport the product between the different types of locations involved in the decision problem. The goal is to find the optimal walk for a single vehicle to pick up all the demands by using the installed plants to drop-off the product. We call this problem the location-routing problem with Drop-Offs and Budget Constraints (DOBC).

\section*{Contributions}

  The main contributions of this work are the following:
  \begin{itemize}
      \item We introduce a new problem, the location-routing problem with Drop-Offs and Budget Constraint, %
      motivated by different practical applications.
      \item We provide a general and flexible mathematical optimization based framework that allows for different practical variants of interest.
      \item We present a mixed integer linear optimization model for the problem that considers all the elements with novel connectivity constraints.
      \item We develop an specialized two-phases branch-and-cut approach for separating the connectivity constraints.
      \item We report the results of an extensive battery of experiments on synthetic datasets that validate our proposal and explore some insights of our proposed decision aid tool.
  \end{itemize}

\section*{Organization}

The rest of the paper is organized as follows. { In Section \ref{S_review} we review the recent literature related to the DOBC problem.} We describe the problem in Section \ref{S_problem} , including its input data and the notation that we use through this paper. We illustrate different variants of the problem, as well as the assumptions that a solution must satisfy. Section \ref{sec:mathopt} is devoted to derive the mathematical optimization models that we propose to solve the problem, and the two-phase branch-and-cut that we develop to solve it. In Section \ref{sec:comput} we report the results of our computational experience. Finally, in Section \ref{sec:conclusions} we draw some conclusions and future lines of research on the topic.

\section{Literature Review}\label{S_review}

The DOBC, as far as we know, has not been previously analyzed in the literature, but it shares some characteristics with widely studied problems. On the one hand, the DOBC has some similarities with the Vehicle Routing Problem (VRP). The Vehicle Routing Problem (VRP) is a classical combinatorial optimization problem first introduced by \citet{dantzig1959truck}, whose goal is to find the least costly routes for a fleet of vehicles to pick up the demand of a set of customers and deliver it to one or more depots. There are many variants of this problem, such as requiring capacities for the vehicles (CVRP-Capacitated VRP), limitations on the service time for the users (VRPTW--VRP with Time Windows), etc.~\citep[see, e.g.][]{TothVigoVRP,survey22}. In the VRP, the routes start and end at the depot. Variants with multiple depots (MDVRP) have also been considered in the literature and, in that case, each demand point is allocated to one of the depots, and closed routes rooted at each of the depots provide service to all demand points allocated to it. When only paths starting at a depot and visiting several customers are sought, the problem is known as Open VRP (OVRP). Extensions of the above problems that allow for multiple visits to the customers are refered to as Split Delivery VRP(SDVRP) \citep[see e.g.][]{split}.  In the DOBC that we propose here, different plants are allocated to the same route, acting as intermediate drop-off points in the route, and additionally, we allow to split the demand, that is, when the capacity is limited, the vehicle can visit several times the same pickup point to load part of the demand until it is completely collected.

Another variant of the VRP that shares some characteristics with the DOBC is the single vehicle routing problem with deliveries and selective pick-ups (SVRPDSP). In the SVRPDSP \citep[][]{gribkovskaia2008single}, all the points act as both, pick-up and delivery points (with a pick-up demand and a delivery demand each), and the goal is to build a route starting at a given depot and picking up and delivering the required demands at all these points. However, as opposite to the DOBC there is a single depot and the designed route must start and end at its location. Other types of problems also require tracing routes by picking up and delivering a product, as in the \emph{one-commodity pick-up-and-delivery traveling salesman problem} (1-PDTSP) introduced in \cite{hernandez2004branch}. In the 1-PDTSP, the goal is to construct a minimum cost circuit for a vehicle to traverse a set of given demand points, where some of them have positive demands (pick-up points) and others have negative demands (delivery points). In the DOBC that we propose, the pick-up points and their demands are given, but the drop-off points are to be decided, and there are no requirements for the received amount at those locations (which will also be decided). There is no restriction about the number of plants that can be used for each pick-up point.  Moreover, in the DOBC, we provide a general framework where, although the route for the vehicle must be feasible to be traversed by a single vehicle, {and its starting and ending points must belong to the set of installed plants, they need not be necessarily the same one} (unless the decision maker explicitly desires this requirement in the design). Other versions of pick-up and delivery problems are proposed, for instance, in \cite{tajik2014robust,gribkovskaia2008one} and \cite{savelsbergh1995general}. Another problem that uses the terminology \textit{pick-up-and-delivery} is the Traveling Salesman Problem with pick-ups and Deliveries (TSPPD), introduced in \cite{dumitrescu2010traveling}, but in this problem there are no demands associated with the users, only paired pick-up-delivery requests, and the goal is to construct a route between the starting depot and the end depot satisfying the requests. So, this can be considered as a multy-commodity problem, as opposite to the DOBC.

None of the above extensions of the VRP and the TSP, considers any location decision. Location analysis is among the most active research areas in Operations Research, and the problems in this discipline consist of finding the positions of one or more entities that optimize some given measure of the service given to a fixed set of demand points. This general definition of a location problem allows one to classify as location problems those that appear in many different areas beyond the location of stores to satisfy the demand of a set of customers (as in the classical $p$-median~\cite{hakimi1964optimum} or the $p$-center~\cite{hakimi1965optimum} problems) but also to problems arising in Machine Learning, such as clustering~\citep{locationForClustering}, supervised classification~\citep{blanco2020lp,marin2022soft}, or fitting of hyperplanes~\citep{hiperplanos,blanco2021multisource}. For a recent monograph on Location Science, the reader is referred to the book \cite{LocationScience19}. Among all the versions of location problems, one can find those that combine location decisions with routing decisions (the VRP and its versions). This family of problems is called Location Routing (LR) problems~\citep[see e.g.][]{albareda2019location,laporte1988location}. In all these problems, the locational decisions are focused on the depots, and among all the available (open) depots, the demand points are typically allocated to one of them. Thus, if the location and allocation decisions are known (open depots and demand points allocated to each of them), the problem results in a VRP problem. In the DOBC, there are no proper depots involved in the problem, but the locational decisions are related to the intermediate drop-off points in the route. An extension of the 1-PDTSP that includes location decisions is considered in \cite{DOMINGUEZMARTIN2024106426}: the one-commodity pickup and delivery location routing problem (1-PDLRP). This is probably the problem in the literature that is closest to the DOBC. However, as opposite to the DOBC, in the 1-PDTSP exaclty one route has to be defined for each facility, that visits all its allocated customers and it is used to balance the demands of the pick-up and delivery points.

Another family of problems that also relates the DOBC is the one of \textit{Steiner Problems}~\citep[][]{dreyfus1971steiner,maculan1987steiner}. In Steiner problems, given a graph and a selected set of nodes of the graph, the goal is to construct a minimum cost structure (tree, path, cycle) that visits, at least, this selected set of nodes. In the DOBC, beyond the requirements of picking up the product from a set of points, the goal is to construct a route for a vehicle that contains, at least,  these pick-up points, but only some of others (the drop-off points) that will be intermediate points in the route. However, in contrast to Steiner problems, the drop-off points will be required to assure, not only a minimum-cost solution but also feasible solution with respect to the vehicle capacities. In particular, in the DOBC, at least one drop-off point is required to unload the demand picked up in the demand positions. Steiner problems have been widely analyzed in the literature, specially the Steiner Tree problem~\cite[see e.g.][among others]{ljubic2021solving,maculan2000euclidean,fischetti2017thinning,rodriguez2019steiner}.

We present a general mathematical optimization framework for the DOBC, allowing the decision maker to choose different characteristics for the problem, with some interesting particular cases. It is worth mentioning that, unlike other pick-up and delivery problems, we do not assume unsplitable demands, i.e., our model allows multiple visits to the pick-up and drop-off points. This situation is particularly useful in cases where the vehicle has limited capacity. Nevertheless, the decision maker provides an upper bound for the  number of visits to each point. If this upper bound is set to one for all the points, then a single visit is required. On the other hand, in the DOBC, as already mentioned, we do not assume the existence of a given depot. One can see the drop-off points as the depots, whose locations also have to be decided. However, the start and end drop-off points for the route must also be decided in this problem, although if they are known, they can be fixed before solving it, simplifying the decisions. Finally, we do not assume that the route to be traced by the vehicle is a tour with the same start and end point, but that the route is a generalized Eulerian walk, in the sense that it starts at a drop-off point and ends at a (potentially different) drop-off point. 

Table~\ref{tab:problemas} summarizes the main features of the problems related to the DOBC. 
\begin{table}[h!]
  \centering
    {\small\begin{tabular}{lccccccc}
          & \multicolumn{3}{c}{plants (drop-off points)} & \multirow{2}[3]{*}{n. routes} & \multicolumn{1}{c}{demand} & \multicolumn{1}{c}{route} & split \\
\cmidrule{2-4}          & fixed? & number & \#/route &      &   type     & topology & demands?\\
    \midrule
    \textbf{DOBC}  & no    & decide & decide & 1     & pick-up & flexible & yes\\
    VRP   & yes   & 1     & 1     & decide & deliver & closed  & no\\
    MDVRP & yes   & fixed & 1     & decide & deliver & closed & no\\
    SDVRP & yes   & fixed & 1     & decide & deliver & closed & yes\\
    OVRP  & yes   & fixed & 1     & decide & deliver & p-c path & no\\
    SVRPDSP & yes   & 1     & 1     & 1     & both  & closed & no\\
    1-PDTSP &  -    & 0     & 0     & 1     & either & closed & no\\
    TSPPD & yes   & 1     & 1     & 1     & by pairs & closed & no\\
    LRP   & no    & decide & 1     & decide & deliver & closed & no\\
    1-PDLRP & no & decide & 1 & decide & either & closed & no\\
    \end{tabular}}%
  \caption{Comparison of the DOBC with related problems.\label{tab:problemas}}
\end{table}%

\section{The single vehicle location-routing problem with Drop-Offs and Budget Constraints\label{S_problem}}

In this section we introduce the problem under study and set the notation for the rest of the paper. We present a general and flexible framework for the problem, although we will illustrate some special and interesting particular cases of our approach.

We are given a directed graph $G=(V,A)$, where $V$ is the set of nodes and $A$ is the set of arcs. 
A subset $P \subset V$ is considered as the set of pick-up points, and each of them, $v\in P$,  is endowed with a demand $p_v \geq 0$. A subset $U \subset V$ is also given, indicating the nodes that can be activated as \emph{drop-off} (or unloading) points.  The installation or activation of a drop-off point $w\in U$ has a setup cost $F_w$ and a limited budget, $B$, is available. We assume, without loss of generality, that $P \cup U = V$. Each arc $a \in A$ has associated two costs: a cost for traversing the arc, $C_a >0$; and a transportation cost per unit flow, $C_a'$. We also assume that a single vehicle will collect and unload the demand and that it has a given capacity $\rho>0$. 

The goal of the \textit{location-routing problem with Drop-Offs and Budget Constraints} (DOBC) is to construct a route of \emph{minimun cost} that visits all points in $P$ at least once to pick-up all their demands, and to decide the subset of the potential points in $U$, that can be afforded for the given budget $B$, that must be activated to unload the collected product.

To clarify the requirements to be satisfied by a decision made for this problem, in what follows we list some of the considerations to be taken into account: 
\begin{itemize}
\item The demand collected in the pick-up up points must be unloaded and stored at some of the potential drop-off points, which have to be previously installed.
\item The total demand at the pick-up points must be collected. It means that the vehicle should visit at least once (but possibly more) each pick-up point where the demand is stored to load it and drop it off in one of the installed drop-off points. If the capacity of the vehicle, $\rho$, is smaller than the demand at some of the pick-up points it would imply that multiple visits to the pick-up points are required. Even in case the demands are smaller than the capacity, it may be convenient to visit multiple times the pick-up points since the vehicle is not required to automatically drop-off the product and can visit other pick-up points in its route to the drop-off points.
\item The number of visits allowed to the pick-up points is bounded. Otherwise, the implementation of an excessive number of visits in practical applications would be unrealistic, and a vehicle with larger capacity must be considered. Thus, we assume that node $v\in P$ can be visited at most $m_v \in \mathbb{Z}_{>0}$ times to pick-up its demand and that, therefore, its demand  may be split. 
\item The drop-off points are considered as uncapacitated, that is, they can be visited as many times as desired. This assumption allows one to model most of the practical situations where this problem is useful, which is the case when the storage facilities (drop-off points) have larger capacity than the demand to be collected in the pickup points. This assumption implies that in an optimal decision for the problem there will be no arcs in $A$ linking two drop-off points since there is no need to transport the product at no gain.
\item To prevent, in case the budget $B$ is large, the installation of drop-off points to unload  only a small amount of the demand, a minimum amount of the demand, $\eta>0$, is required to be delivered at the activated drop-off points.

\item The route must be able to be traversed by a single vehicle. Thus, the subgraph obtained in the DOBC must be a $\mathbf{m}$-\textit{spanning walk} for the pickup and activated drop-off points, i.e., a sorted sequence of arcs in $A$ in which the ending point of each arc is the starting point of the next arc in the sequence and where each node  $v \in P$ 
appears as starting or ending node of the arcs, at least once  and at most $m_v$ times and, each installed  $w \in U$ appears  at least once. If desired, the decision maker can require the walk to be a \textit{closed walk}, that is, the starting and ending points of the walk coincide, being the resulting subnetwork a cycle.
\item  A solution of the DOBC is assessed by means of a convex combination of the \emph{design costs} and the \emph{flow costs}. The parameter that represents this trade-off between the two goals is denoted by $\alpha\in [0,1]$ (considering only flow costs for $\alpha=0$  and only design costs for $\alpha=1$).
\end{itemize}

In Table \ref{table_notation} we summarize the parameters that we use for the DOBC.

\begin{table}[h]
{\small\centering\begin{tabular}{cl}
    \texttt{Parameter/Sets} & \texttt{Description}\\\hline
    $G=(V,A)$ & Input Graph where the route is constructed.\\
    $P \subset V$ & Pickup points.\\
    $U \subset V$ & Potential drop-off points.\\
    $\{p_v\}_{v\in P}$ & Demand to be collected at the pickup points.\\
    $\{F_w\}_{w\in U}$ & Set-up cost of the drop-off points.\\
    $\{C_a\}_{a\in A}$ & Cost for traversing the arcs.\\
    $\{C'_a\}_{a\in A}$ & Per unit transportation costs on the arcs.\\
    $\rho$ & Capacity of the vehicle.\\
    $B$ & Budget for installing the drop-off points.\\
    $\{m_v\}_{v\in P}$ & Maximum number of visits to the pickup points.\\
    $\eta$ & Minimum demand to be stored at the activated drop-off points.\\
    $\alpha$ & Parameter defining the convex combination of design and flow costs.\\ \hline
    \end{tabular}}
\caption{Input parameters for the DOBC.\label{table_notation}}
\end{table}

\begin{ex}\label{ex:0}
    We illustrate the DOBC with the toy instance drawn in Figure \ref{fig:00} (left plot). We consider $6$ pick-up points ($a_1, \ldots, a_6$) each of them with a demand of $50$ units. We also consider $5$ potential locations for the drop-off points, $b_1, \ldots, b_5$, squared in the picture. The set up cost for installing the facilities is the same, and we assume that there is budget to open at most $3$ of these points. The vehicle picking up and unloading the product has a maximum capacity of $150$ units.
      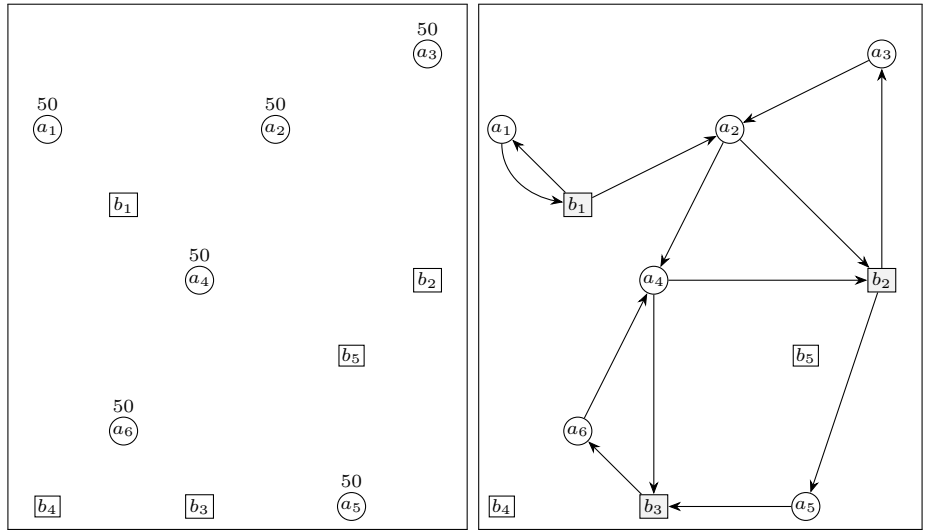
\begin{figure}[h]
\begin{center}
\fbox{
\input{graph_ex_instance}}~\fbox{\input{graph_ex0}}
\end{center}
\caption{Instance for Example \ref{ex:0} (left) and a feasible solution (right).\label{fig:00}}
\end{figure}
In Figure \ref{fig:00} (right plot) we show a feasible (not necessarily optimal) solution for the DOBC. Note that we do not require the route to be a closed walk. The route in this case can be represented as the path:
{\small
$$
\squared{$b_1$} \longrightarrow \stackrel{\text{\tiny 50}}{\circled{$a_1$}} \stackrel{50}{\longrightarrow} \squared{$b_1$} \longrightarrow \stackrel{\text{\tiny 25}}{\circled{$a_2$}} \stackrel{25}{\longrightarrow} \squared{$b_2$} \longrightarrow \stackrel{\text{\tiny 50}}{\circled{$a_3$}} \stackrel{50}{\longrightarrow} \stackrel{\text{\tiny 25}}{\circled{$a_2$}} \stackrel{75}{\longrightarrow} \stackrel{\text{\tiny 25}}{\circled{$a_4$}} \stackrel{100}{\longrightarrow} \squared{$b_2$} \longrightarrow \stackrel{\text{\tiny 50}}{\circled{$a_5$}} \stackrel{50}{\longrightarrow} \squared{$b_3$} \longrightarrow \stackrel{\text{\tiny 50}}{\circled{$a_6$}} \stackrel{50}{\longrightarrow} \stackrel{\text{\tiny 25}}{\circled{$a_4$}} \stackrel{75}{\longrightarrow} \squared{$b_3$}
$$}
where arrows represent the sequence of visited nodes, the text over the arrows indicates the load of the vehicle in this arc, and the text over the pick-up nodes indicates the amount of the total demand picked up in the visit. 

Note that both the pick-up nodes, $P=\{a_1,\cdots , a_6\}$, and the activated drop-off points, $\{b_1, b_2, b_3\}\subset U$, are allowed to be visited multiple times in the trip. Thus, the demand of the pick-up nodes can be split (as in nodes $a_2$ and $a_4$). 
\end{ex}

\begin{thm}
    The DOBC is NP-hard.
\end{thm}
\begin{proof}
For $\rho > \sum_{v\in P} p_v$, $B = \min_{w\in U} F_w$ and $\alpha=1$, the problem reduces to the traveling salesman problem (TSP) for the points in $P \cup \{w\}$ where $w \in \arg\min_{w\in U} F_w$, which, it is well known, is an NP-hard problem.
\end{proof}

\section{Mathematical Optimization Model for the DOBC}\label{sec:mathopt}

In this section we describe the mathematical optimization model that we propose for the problem. First of all, we describe a re-parameterization of the input graph to allow for a route with multiple visits to the pickup points. Then, we derive a mathematical optimization formulation for the problem, by describing the decision variables and constraints. The model will require the incorporation of exponentially many constraints, and then, we develop a 2-phase branch-and-cut approach to incorporate them on-the-fly.

\subsection{The extended graph}

As already mentioned, multiple visits are allowed to the pick-up points, with at most $m_v \in \mathbb{Z}_{>0}$ visits permitted for each $v \in P$. The demand at each point is then split among the different visits. To control the proportion of demand collected during the visits to a pick-up point, we construct an extended graph of $G$ by replicating the vertices in $P$, and consequently their incident arcs, according to the maximum number of allowable visits to each node. The vertices in $U$ do not need to be replicated, as drop-off points are uncapacitated, and the entire product accumulated in the vehicle is unloaded at each visit. Nevertheless, by convention we extend the number of visits to all the nodes in $V$ by assuming that $m_w=1$ for all $w\in U$.

 To construct the extended graph of $G$, node $v\in V$ is replicated $m_v$ times  and arc $a \in A$ in the form $a=(v,w)$ with $v, w \in V, \; v\neq w$, is also replicated $m_v \times m_w$ times, that is, each copy of node $v$ is linked to all copies of node $w$. In what follows, we will denote by $\tP$ the set of replicated pick-up nodes, and by  $\tG=(\tV,\tA)$ the resulting extended directed graph, where $\tV= \tP\cup U$  is the set of replicated vertices and $\tA$ is the set of replicated arcs. In case a single visit is allowed for all the pick-up points, $\widetilde{G}$ coincides with $G$. Thus, from now on, the graph that we use in our model is $\tG$. We denote by $\pi: \tG \rightarrow G$ the \emph{forgetful mapping} that maps the nodes and arcs in the extended graph to the nodes and arcs in the original graph (aggregating the replicas). 

 Note that the graph $\tG$ constructed from $G$ depends on the number of visits allowed to each  point ($m_v$, for $v\in V$), although we do not incorporate this in the notation to simplify it.

Given any set $S \subseteq \tV$ we denote by $\delta(S)$ its cut-set: $\delta(S)=\{(v,w) \in \tA: v \in S \text{ or } w \in S\}$. In case $S=\{v\}$, we denote $\delta(\{v\})=\delta(v)$. Analogously we denote the sets $\delta^+(S)=\{(v,w) \in \tA: v \in S, w \not \in S\}$, $\delta^-(S)=\{(v,w) \in A: w \in S, v \not \in S\}$.

\subsection{Variables}

In what follows we describe the decision variables that we use in our model.

For the sake of describing the routes picking up and dropping off  the product we use two different sets of variables: (a) design variables that determine which nodes and links between two pick-up points, or between pick-up points and drop-off points are activated; and (b) flow  variables that determine the amount of product that is distributed through the different links.  

The first set of variables within the design variables determines the usage of the nodes in $\tilde{G}$, including both the (replicated) pick-up points and the drop-off points.
$$
y_v = \begin{cases}
    1 & \mbox{if $v$ is used in the route.}\\
    0 & \mbox{otherwise}
\end{cases} \ \ \forall v \in \tV.
$$

Note that if $y_w=1$ with $w\in U$, the above variable indicates that the potential drop-off point is activated. Otherwise, if $v\in \tP$, this variable gives us information about which \emph{replica} of the pick-up node is used in the route (visits to the node). For a given pickup point in the original set of pickup points $P$, we call \emph{active} visit to the extended set of pickup points that are used in the walk of the vehicle.

The next decisions that we have to make are on the usage of the arcs in $\widetilde{A}$ in the pick-up and drop-off spanning walk. These decisions are determined in our model by the following sets of design variables:
$$
x_{a} = \begin{cases}
    1 & \mbox{if the walk uses arc $a$},\\
    0 & \mbox{otherwise}
\end{cases} \ \ \forall a \in \tA
$$
Note that, with our extended graph these variables will take value one if both the arc in the original graph is activated and it is linking the corresponding active replicas.

Additionally, the activated links will be used to distribute the product. The amount of product transported through each active arc will be determined by the following flow variables:
$$
f_{a} = \text{amount of demand transported through arc $a$}, \ \ \forall a \in \tA.
$$
Finally, we consider a variable that determines how the demand of a pick-up point is split within the different visits:
$$
q_v \in [0,1]: \text{proportion of demand $p_{\pi(v)}$ collected in the visit to $v$} \ \ \forall v\in \tP.
$$

\subsection{Objective Function}

 With the above sets of variables, we construct the following measures of the transportation costs that we will consider in our model:
\begin{itemize}
    \item {\bf Design Cost}. The overall length/cost of the constructed network is the fixed cost for routing the vehicle:
    $$
    \Gamma(x) = \sum_{a \in \tA}  C_{a} x_{a}
    $$
    \item {\bf Flow Cost}. The amount of product at the vehicles running each arc may affect the cost of the route that is: 
     $$
    \Xi(f) = \sum_{a \in \tA}  C'_a f_{a}
    $$  
\end{itemize}
Both costs are combined in the single objective measure $\alpha \Gamma(x) +  (1-\alpha) \Xi(f)$, where $\alpha\geq 0$ is the weight that defines the trade-off between the design and the  flow cost.

Note that, depending on the particular application of the DOBC, a different value of $\alpha$ must be considered. In case the demand to be transported is not excessive and the transportation costs for the arcs are small, the costs per unit flow would be negligible compared to the design costs. Thus, a value of $\alpha$ closer to $1$ is more convenient. Otherwise, if the product to be transported highly affects the cost of the vehicle along the walk, a value closer to $0$ would be more desirable. In general, an intermediate value of $\alpha$ allows to consider both factors in the optimal decision.

\subsection{Constraints}

The above sets of variables are linked via the following sets of linear inequalities to assure their correct definition as well as the technical requirements of the problem:
\begin{itemize}
    \item The limited budget to install drop-off points is enforced by the following inequality:
    \begin{equation}
    \sum_{w \in U} F_w y_w \leq B.\label{mC:ctr1}
    \end{equation}
    Observe that in case all the drop-off points have the same set-up cost, i.e., $F_w=F$, for all $w\in U$, the above inequality reduces to:
    \begin{equation}
        \sum_{w\in U} y_w \leq p := \left\lceil\frac{B}{F}\right\rceil, \label{mC:ctr2}
    \end{equation}
    providing an upper bound for the number of drop-off points to open.
    \item No links are allowed incident to non-open drop-off points:
    \begin{align*}
        x_a \leq y_w, \forall w \in U,  a \in \delta(w).
    \end{align*}
These constraints can be aggregated as follows:
\begin{align}
\sum_{a\in \delta(w)}  x_{a} \leq 2|\tP|\, y_w, \forall w \in U.\label{mC:ctr3}
\end{align}
In case a potential drop-off point is not considered to unload the product, all the variables indicating links to them (incoming from pick-up points and outgoing to pick-up points) are set to zero.
\item The whole demand of each pick up point is split among the active visits to the pickup points:
\begin{align}
    \sum_{v' \in \tP: \pi(v')=v} q_{v'} = 1, \forall v\in P, \label{mC:ctr4}
    \end{align}
It is also required to enforce that the vehicle is not allowed to pick-up any amount of demand from a replica of a pickup point that is not active:
\begin{align}
     q_v \leq y_v, \forall v \in \tP.\label{mC:ctr5}   
\end{align}
and that the pick-up nodes must be visited at least once:
\begin{align}
    \sum_{v' \in \tP: \pi(v')=v} y_{v'}\geq 1, \forall v \in P. \label{mC:ctr5b}
\end{align}
    \item There is a single outgoing arc and a single incoming arc for the active visits to the pick-up points:
\begin{align}
     \sum_{a \in \delta^+(v)} x_{a} &= y_v,  \forall v \in \tP, \label{mC:ctr6a}\\
      \sum_{a \in \delta^-(v)} x_{a} &= y_{v}, \forall v \in \tP. \label{mC:ctr6b}
\end{align}
Note that  the left hand side expressions indicate the number of outgoing/incoming arcs from/to  the nodes in $\tP$. In case $y_v=1$, the above expressions enforce that a single arc emanates from/to the node $v$ (an active replica of the original node $\pi(v)$). In case $y_v=0$, no arcs connecting this replica are  allowed.
\item Flow balancing constraints at the replicas of pick-up points state that the outgoing amount of product  at a pick-up point must coincide with the incoming amount at that point plus the split of the demand collected at that active visit:
\begin{align}
     \sum_{a \in \delta^+(v)} f_{a} - \sum_{a \in \delta^-(v)} f_{a} = q_v p_{\pi(v)}, \forall v \in \tP.\label{mC:ctr7}
\end{align}
\item The product is limited to be transported only through the active arcs, and in case it is active it is bounded above by the capacity of the vehicle:
\begin{align}
& f_{a} \leq \rho x_{a}, \ \ \forall a \in \tA\label{mC:ctr8a}.
\end{align}
\item A minimum amount of demand is required to be received at the open drop-off points:
\begin{align}
 \sum_{a \in \delta^-(w)}   f_{a} \geq \eta y_{w}, \ \ \forall  w \in U, \label{mC:ctr9}
\end{align}
Note that in case the budget $B$ is large, one might install drop-off points to store only a small amount of the product. The above inequality prevents this situation.
\item A single vehicle must be able to traverse all active pick-up and drop-off points in the route. This is assured by imposing that the resulting network is an Eulerian walk, which is equivalent to enforce that there is at most a pair of drop-off points with net degree ($|$in-degree $-$ out-degree$|$) equal to one and it is zero for all the other vertices. By Constraints \eqref{mC:ctr6a} and \eqref{mC:ctr6b}, each of the replicas of the  pick-up points has $0$ net degree  ($|0-0|$ if it is not active  or $|1-1|$, if it is active). We only impose this condition to the active drop-off points. 

To do so, we incorporate the following two additional sets of variables:
$$
e_w = \begin{cases}
    1 & \mbox{if drop-off point $w$ has in-degree $>$ out-degree},\\
    0 & \mbox{otherwise}
\end{cases}\quad \forall w \in U,
$$
$$
s_w = \begin{cases}
    1 & \mbox{if drop-off point $w$ has in-degree $<$ out-degree},\\
    0 & \mbox{otherwise}
\end{cases} \quad \forall w \in U.
$$
These variables are adequately defined by the following equations:
\begin{align}
    & \sum_{a \in \delta^+(w)} x_{a} -\sum_{a \in \delta^-(w)}   x_{a} = s_w-e_w, \forall w \in U, \label{mC:ctr10}
    \end{align}
Using these variables, we impose that the resulting graph is Eulerian or semi-Eulerian with the following constraints:
\begin{align}
& \sum_{w \in U} e_w\leq 1,  \label{mC:ctr11a}\\
& \sum_{w \in U} s_w\leq 1,  \label{mC:ctr11b}\\
& \sum_{w \in U} e_w- \sum_{w \in U} s_w= 0,  \label{mC:ctr11c}\\
& e_w + s_w \leq y_w,\; \forall w\in U, \label{mC:ctr11d}
\end{align}
Constraints \eqref{mC:ctr11a} and \eqref{mC:ctr11b} impose that there is at most one drop-off point with in-degree $>$ outdegree and at most one drop-off point with out-degree $>$ in-degree, respectively.  Furtheremore, Constraint \eqref{mC:ctr11c} enforces that in case there is an drop-off point with  out-degree $>$ in-degree, then there must be an drop-off point with in-degree $>$ out-degree. Indeed, these constraints are not really necessary, since these are the only nodes with odd-degree in the graph, and their number must be even, but we include them for completeness of the formulation. Finally, by Constraints \eqref{mC:ctr11d} these two points must be different. Then, the resulting network is an Eulerian walk. The starting and ending points of the walk can be detected using these variables. In case $e_w=1$, for some $w\in U$, the drop-off point $w$ is the ending point in the walk, whereas if $s_w=1$, then the drop-off point is the starting point in the walk.

In case all the drop-off points have in-degree $=$ out-degree, the resulting network is an Eulerian cycle, otherwise, the network is an Eulerian trail. This condition can be imposed avoiding the use of these variables, as we will explain in subsection \ref{versions}.
\end{itemize}

\subsection{Connectivity Constraints}

When solving the DOBC, the resulting directed network (defined  by the $x$- an $y$-variables) must be \emph{weakly} connected in order to be traversed by a single vehicle. 
It should be possible to trace a path from the starting to the ending drop-off location visiting all the activated vertices of the extended graph.

Imposing connectivity of a network derived from an optimization problem has been widely studied in the literature, and there are different approaches to incorporate this condition in the shape of linear inequalities (see e.g. \cite{frank1995connectivity}). We adopt a cut-set based methodology~\cite{boyd1993integer,partition2000} to impose the connectivity of the resulting network. This type of approaches have been successfully applied to several network design problems~\cite{grotschel1992computational,hubsVEY2023,corberan2021distance}.
In our case, we propose a family of cut-set based inequalities to enforce the solution digraph to be connected, and then feasible to route the product through its vertices with a single vehicle. Since these constraints are exponentially many, we will provide in Section \ref{S:branch-cut} a separation approach to detect its violation for a feasible solution (when these constraints are relaxed).

The idea behind our connectivity constraints is that given any   subset of nodes, $S \subsetneqq \tV$, such that both, set $S$ and its complement set, $S^c$, are not empty sets in the directed solution network,  if the starting  point of the route  (i.e.,  drop-off point $w\in U$ such that, $s_w=1$)  is not in  $S^c$ (because either, the starting point is in $S$, or the resulting route is a cycle and there is not any starting point) then, there must be at least an activated arc 
from a node in $S$ to a node in $S^c$.

\begin{prop}
The following constraints ensure the connectivity of the subnetwork defined by a solution:
\begin{align}
\sum_{a \in \delta^+(S)}  x_a \geq y_v+ y_{v'}-1-\sum_{w\in S^c\cap U} s_w, 
\quad  \forall S\subset \tV,\; v \in S\cap \tP,\; v' \in S^c\cap\tP.\label{mC:ctr12bb}
\end{align}
\end{prop}
\begin{proof}
The right-hand side  can take a strictly positive value only when $v$ and $v'$ are active nodes used in the route (which means that  sets $S$ and $S^c$ are non-empty) and the starting  point of the route (if any) does not belong to $S^c$. In this case, its value is $1$ and  the inequality imposes that there must be at least an active arc going from  a vertex  in $S$  to a vertex in $S^c$. 

Notice that to assure the connectivity is sufficient to consider nodes $v, v'$ in $\tP$ since  Constraints \eqref{mC:ctr8a} and \eqref{mC:ctr9} assure the connectivity from nodes in $\tP$ and   any activated drop-off point.
\end{proof}

\subsection{The model\label{S: model}}
Summarizing, the model we propose for the DOBC is the following:
\begin{align}
\min & \quad\alpha \Gamma(x) +  (1-\alpha) \Xi(f)\nonumber\\
\mbox{s.t.} & \quad \sum_{w \in U} F_w y_w \leq B,&\eqref{mC:ctr1}\nonumber\\
 &\sum_{a\in \delta(w)} x_{a} \leq 2|\tP|  y_w, \forall w \in U, &\eqref{mC:ctr3}\nonumber\\
 &    \sum_{v'\in \tP: \pi(v')=v} q_{v'} = 1, \forall v\in P, &\eqref{mC:ctr4}\nonumber\\
 &     q_v \leq y_v, \forall v \in \tP, &\eqref{mC:ctr5}  \nonumber \\
 & \sum_{v'\in \tP: \pi(v')=v} y_{v'}\geq 1, \forall v \in P. &\eqref{mC:ctr5b}\nonumber\\
 &\sum_{a \in \delta^+(v)}  x_{a} = y_{v},  \forall v \in \tP, &\eqref{mC:ctr6a}\nonumber\\
  &      \sum_{a \in \delta^-(v)} x_{a}= y_{v}, \forall v \in \tP, &\eqref{mC:ctr6b}\nonumber\\
  &     \sum_{a \in \delta^+(v)}  f_{a} - \sum_{a \in \delta^-(v)}  f_{a} = q_v p_{\pi(v)}, \forall v \in \tP,&\eqref{mC:ctr7}\nonumber\\
  & f_{a} \leq \rho x_{a}, \ \ \forall a \in \tA,  &\eqref{mC:ctr8a}\nonumber\\
 &  \sum_{a \in \delta^-(w)}  f_{a} \geq \eta y_{w}, \ \ \forall  w \in U, &\eqref{mC:ctr9}\nonumber\\
     & \sum_{a \in \delta^+(w)} x_{a} -\sum_{a \in \delta^-(w)}   x_{a} = s_w-e_w, \forall w \in U,  &\eqref{mC:ctr10}\nonumber\\
 & \sum_{w \in U} e_w\leq 1,  &\eqref{mC:ctr11a}\nonumber\\
 & \sum_{w \in U} s_w\leq 1,  &\eqref{mC:ctr11b}\nonumber\\
 & \sum_{w \in U} e_w- \sum_{w \in U} s_w= 0,  &\eqref{mC:ctr11c}\nonumber\\
 & e_w + s_w \leq y_w,\; \forall w\in U, &\eqref{mC:ctr11d}\nonumber\\
 & \sum_{a \in \delta^+(S)}  x_a \geq y_v+ y_{v'}-1-\sum_{w\in S^c\cap U} s_w
\quad  \forall S\subset \tV,\; v \in S\cap \tP,\; v' \in S^c\cap\tP, &\eqref{mC:ctr12bb}\nonumber\\
& y_v \in \{0,1\}, \forall v \in \tV,&\label{V:domy}\\
& q_v \in [0,1], \forall v \in \tP,&\label{V:dompi}\\
& s_w, e_w \in \{0,1\}, \forall w \in U,&\label{V:domU}\\
& x_a \in \{0,1\}, \forall a \in \tA, &\label{V: domx}\\
& f_{a} \geq 0, \forall a \in \tA,& \label{V: domf}
\end{align}

\subsection{Special versions of the problem \label{versions}}

In what follows we detail different versions of the problem by slightly modifying the model above. Note that although our model does not impose the resulting walk to be closed (a cycle) it can be also enforced, as well as restricting the number of visits to the pick-up points or drop-off points.

Different versions of the DOBC can be derived by imposing the model the different conditions, namely, a common maximum number of visits to the drop-off points, $k_d$, a common maximum  number of visits to the pickup points, $k_p=m_v$ for all $v \in P$, or the topology of the resulting route, ${T} \in \{C,P\}$, where $C$ indicates a cycle and  $P$ a path (it could eventually be a cycle but is not explicitly required) . We denote by $(k_d,k_p)$-${\rm DOBC}\_{T}$ this family of special cases. Some interesting problems within this family are:

\begin{itemize}
    \item If $T = C$, that is, if the resulting subnetwork is required to be a cycle, it can be enforced by fixing all the variables $e_w$ and $s_w$ in the model above to take value $0$, which is equivalent to replace Constraints \eqref{mC:ctr10}-\eqref{mC:ctr11d} by:
    \begin{align}
    & \sum_{a\in \delta^+(w)} x_a  -\sum_{a\in \delta^-(w)} x_a  = 0, \forall w \in U, \label{mC:cycle1}
    \end{align}
    and Constraints \eqref{mC:ctr12bb} by:
    \begin{align}
     \sum_{a \in \delta^+(S)}  x_a \geq y_v+ y_{v'}-1, \forall S \subset \tV, v \in S\cap \tP, v' \in S^c\cap \tP. \label{mC:cycle2}
     \end{align}
     \item If $k_p=1$, a single visit to each pick-up point is allowed. Thus, the extended graph $\tG$ coincides with the original graph $G$, in particular, $\tP=P$, and neither $y_v$  nor $q_v$ variables need to be defined  for any $v\in P$, since demands are not split. Then, Constraints \eqref{mC:ctr4}, \eqref{mC:ctr5} and \eqref{mC:ctr5b} can be eliminated since they are no longer needed and   the values of the $y$ or $q$ variables can be fixed to $1$ in Constraints \eqref{mC:ctr6a},  \eqref{mC:ctr6b}, \eqref{mC:ctr7} and   \eqref{mC:ctr12bb}. In the last ones, only sets $S$ with $S\cap P \neq \emptyset$ and $S^c\cap P \neq \emptyset$ must be considered.
    \item If $k_d=1$, a single visit to each drop-off point is allowed, and then,  Constraints \eqref{mC:ctr3} must be replaced by the following constraints:
    \begin{align}
     &\sum_{a \in \delta^+(v)}  x_{a} \leq y_{w},  \forall w \in U, \label{mA:ctrk1}\\
  &      \sum_{a \in \delta^-(v)} x_{a}\leq y_{w}, \forall w \in U. \label{mA:ctrk2}
    \end{align} 
    { Note that if $\eta >0$ Constraints  \eqref{mA:ctrk1} and  \eqref{mA:ctrk2} together with Constraints \eqref{mC:ctr8a}-\eqref{mC:ctr11d} force the resulting network to be a cycle.}
\end{itemize}

\begin{ex}\label{ex:1}
We will use the same instance of Example \ref{ex:0} to show feasible (not necessarily optimal) solutions to  different versions of the problem. 

In Figure \ref{fig1} we show feasible solutions for the problems $(\infty,2)-{\rm DOBC}\_C$ (Figure \ref{cycle}), where the resulting route is forced to be a cycle, $(\infty,1)-{\rm DOBC}\_P$ (Figure \ref{1P}),  where a single visit to each pick-up point is allowed, $(1,2)-{\rm DOBC}\_P$ (Figure \ref{1U}), where a single visit to each drop-off point is allowed, and $(1,1)-{\rm DOBC}\_C$ (Figure \ref{all}),  where all the above conditions are imposed together.
\begin{figure}[h]
\centering
\subcaptionbox{$(\infty,2)$-{\rm DOBC}\_C\label{cycle}}[6.5cm]{\fbox{\input{graph_ex1}}}
\subcaptionbox{$(\infty,1)$-{\rm DOBC}\_P\label{1P}}[6.5cm]{\fbox{\input{graph_ex2}}}\\
\subcaptionbox{$(1,2)$-{\rm DOBC}\_P\label{1U}}[6.5cm]{\fbox{\input{graph_ex3}}}
\subcaptionbox{$(1,1)$-{\rm DOBC}\_C\label{all}}[6.5cm]{\fbox{\input{graph_ex4}}}
\caption{Feasible solutions (non necessarily optimal) the different versions of the problem. \label{fig1}}
\end{figure}
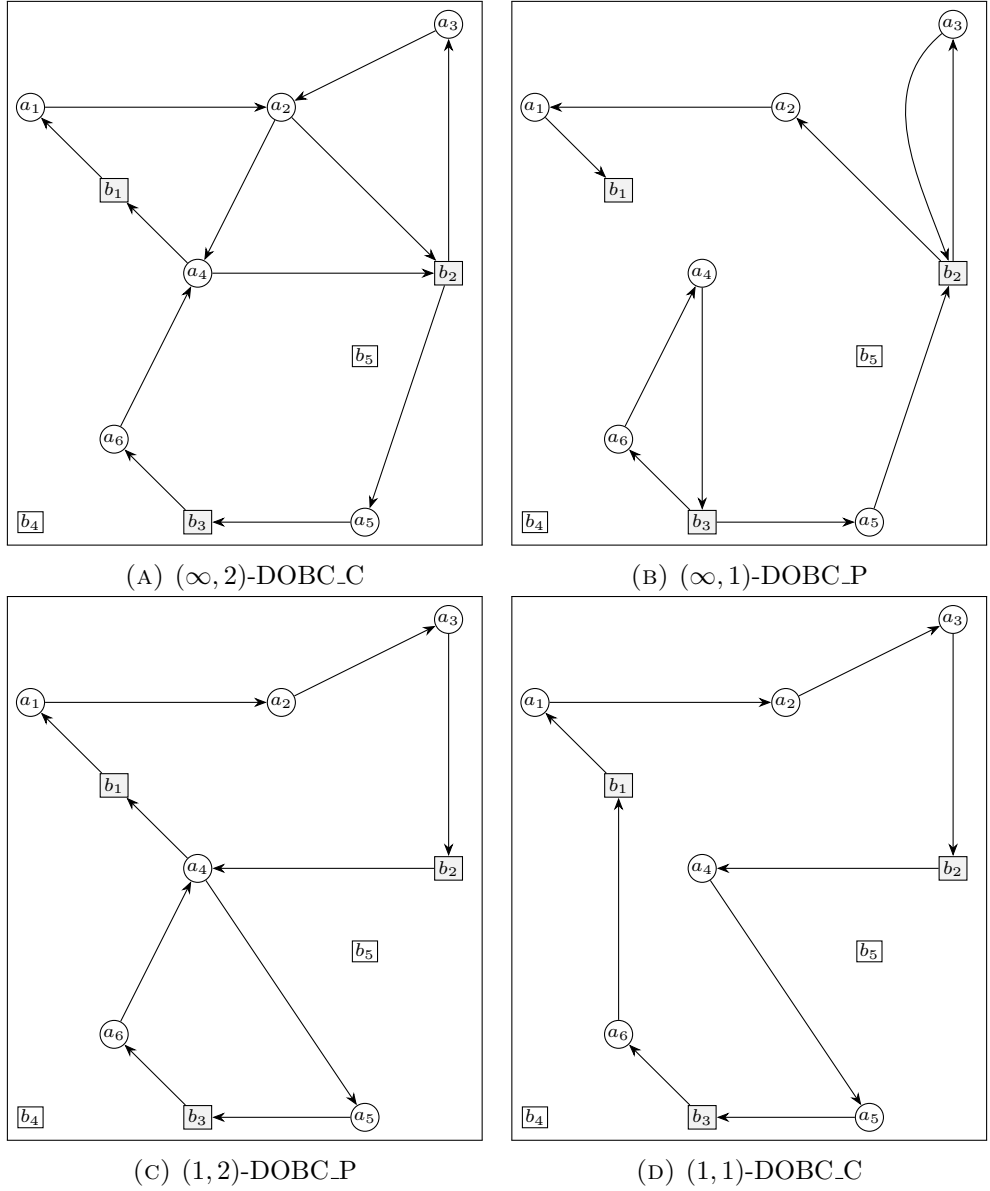

Specifically, the walks of the vehicle in the different solutions are:
\begin{description}
    \item[\small$(\infty,2)$-{\rm DOBC\_C}: ] {\small
$\squared{$b_3$} \rightarrow \stackrel{\text{\tiny 50}}{\circled{$a_6$}} \stackrel{50}{\longrightarrow} \stackrel{\text{\tiny 25}}{\circled{$a_4$}} \stackrel{75}{\longrightarrow} \squared{$b_1$}  \rightarrow \stackrel{\text{\tiny 50}}{\circled{$a_1$}}  \rightarrow \stackrel{\text{\tiny 25}}{\circled{$a_2$}} \stackrel{75}{\longrightarrow} \squared{$b_2$} \rightarrow \stackrel{\text{\tiny 50}}{\circled{$a_3$}} \stackrel{50}{\longrightarrow} \stackrel{\text{\tiny 25}}{\circled{$a_2$}} \stackrel{75}{\longrightarrow} \stackrel{\text{\tiny 25}}{\circled{$a_4$}} \stackrel{100}{\longrightarrow}  \squared{$b_2$}  \rightarrow \stackrel{\text{\tiny 50}}{\circled{$a_5$}} \stackrel{50}{\longrightarrow}  \squared{$b_3$}
$}  
    \item[\small$(\infty,1)$-{\rm DOBC\_P}: ] {\small
$\squared{$b_3$} \rightarrow \stackrel{\text{\tiny 50}}{\circled{$a_6$}} \stackrel{50}{\longrightarrow} \stackrel{\text{\tiny 50}}{\circled{$a_4$}} \stackrel{100}{\longrightarrow} \squared{$b_3$}  \rightarrow \stackrel{\text{\tiny 50}}{\circled{$a_5$}}  \stackrel{50}{\longrightarrow} \squared{$b_2$} \rightarrow \stackrel{\text{\tiny 50}}{\circled{$a_3$}} \stackrel{50}{\longrightarrow} \squared{$b_2$} \rightarrow \stackrel{\text{\tiny 50}}{\circled{$a_2$}} \stackrel{50}{\longrightarrow} \stackrel{\text{\tiny 50}}{\circled{$a_1$}} \stackrel{100}{\longrightarrow} \squared{$b_1$}
$}  
    \item[\small$(1,2)$-{\rm DOBC\_P}: ] {\small
$
\squared{$b_1$} \longrightarrow \stackrel{\text{\tiny 50}}{\circled{$a_1$}} \stackrel{50}{\longrightarrow} \stackrel{\text{\tiny 50}}{\circled{$a_2$}} \stackrel{100}{\longrightarrow}  \stackrel{\text{\tiny 50}}{\circled{$a_3$}}  \stackrel{150}{\longrightarrow} \squared{$b_2$} \longrightarrow \stackrel{\text{\tiny 25}}{\circled{$a_4$}} \stackrel{25}{\longrightarrow} \stackrel{\text{\tiny 50}}{\circled{$a_5$}} \stackrel{75}{\longrightarrow} \squared{$b_3$} \longrightarrow \stackrel{\text{\tiny 50}}{\circled{$a_6$}} \stackrel{50}{\longrightarrow} \stackrel{\text{\tiny 25}}{\circled{$a_4$}} \stackrel{75}{\longrightarrow} \squared{$b_1$}
$}
    \item[\small$(1,1)$-{\rm DOBC\_C}: ] {\small
$
\squared{$b_1$} \longrightarrow \stackrel{\text{\tiny 50}}{\circled{$a_1$}} \stackrel{50}{\longrightarrow} \stackrel{\text{\tiny 50}}{\circled{$a_2$}} \stackrel{100}{\longrightarrow}  \stackrel{\text{\tiny 50}}{\circled{$a_3$}}  \stackrel{150}{\longrightarrow} \squared{$b_2$} \longrightarrow \stackrel{\text{\tiny 50}}{\circled{$a_4$}} \stackrel{50}{\longrightarrow} \stackrel{\text{\tiny 50}}{\circled{$a_5$}} \stackrel{100}{\longrightarrow} \squared{$b_3$} \longrightarrow \stackrel{\text{\tiny 50}}{\circled{$a_6$}} \stackrel{50}{\longrightarrow}  \squared{$b_1$}
$}
    
\end{description}
Note that the subnetwork obtained for problem $(1,2)$-{\rm DOBC}\_P is a cycle, even though we do not imposed it. Note also that all the solutions are feasible for a vehicle capacity $\rho=150$, but in case the capacity were  smaller, the solutions for $(1,2)$-{\rm DOBC}\_P and $(1,1)$-{\rm DOBC}\_C would no longer be feasible, since the vehicle load is $150$ in some arcs.
\end{ex}
 
\subsection{Symmetry breaking constraints}

Note that the problem exhibits symmetry in the set of visits of the same pickup point, and one can obtain an equivalent solution by re-sorting them in any order. To avoid this situation which may be inconvenient for any implicit enumeration solution algorithm, we impose the following constraints:
\begin{align} 
y_{v_1}&=1\\
 y_{v_{i+1}}& \leq y_{v_i} \quad \forall i=1,\cdots ,m_{v-1}\\
 q_{v_{i+1}}& \leq q_{v_i} \quad \forall i=1,\cdots ,m_{v-1}
 \end{align}
 That is, the first replica of vertex $v\in P$, $v_1$, is always activated, and one cannot use the $(i+1)$-th replica unless the $i$-th replica is also used. Since replicas can still be resorted, we assign the label $i$ to the replica with $i$-th largest proportion of demand picked up in that visit.

\subsection{A two-phase Branch-and-Cut approach\label{S:branch-cut}}

As already mentioned, in our model, the size of the family of Constraints \eqref{mC:ctr12bb} 
is exponential in the number of vertices of $\tG$. 
It is thus not possible to solve the formulation  directly with some off-the-shelf solver, even for medium size instances. In  what follows we present an exact branch-and-cut algorithm for this formulation in which this complex set of constraints is initially relaxed and the constraints are incorporated to the model as needed in the solution procedure. The strategy that we describe below is embedded within an enumeration tree and it is applied not only at the root node but also at all explored nodes.

Our separation procedure is an adaptation of the separation procedure for classical connectivity constraints~\citep{PG-1985} that has been successfully applied in other types of network design problems  \citep[see, e.g][]{hubsVEY2023}.

Initially, constraints \eqref{mC:ctr12bb} are not incorporated to the model. Some of them will be incorporated as needed, until we are sure that connectivity of the resulting network is verified. 
 The goal of the separation approach is to find $S \subset \tV$ such that \eqref{mC:ctr12bb} is violated, that is:
$$
\sum_{a \in \delta^+(S)} x_a < y_v+ y_{v'}-1-\sum_{w\in S^c\cap U} s_w,
$$
for some $v \in S \cap \tP$, $v'\in S^c \cap \tP$.

Let $\bar y, \bar x, \bar s$ be the (possibly fractional) solution values of the linear relaxation of the problem in the design variables involved in the connectivity constraint. To separate the connectivity constraints, we proceed in two phases:
\begin{itemize}
    \item {\bf Phase I}: To avoid overloading the model with constraints constructed over the entire extended graph, we first separate a relaxed version of the connectivity constraints, where all the replicas of the original nodes are aggregated, that is:
    \begin{align}
\sum_{a \in \delta^+(\tilde{S}): \atop \tilde{S}\subset \tilde{V}, \pi(\tilde{S})=S}  x_a + \sum_{w\in S^c\cap U} s_w\geq 1, 
\quad  \forall S\subset V,\; S\cap P\neq \emptyset,\; S^c\cap P \neq \emptyset.\label{mC:connectivity_relax}
\end{align}
Note that these constraints result from aggregating the constraints \eqref{mC:ctr12bb} by all the visits (replicas) of the arcs emanating from the sets $S$ and $S^c$. It enforces that from every subset of the original node set $V$, whose complement contains no starting drop-off point for the walk, at least one arc must emanate from one of the replicas of the nodes in that subset.

To this end, we construct the undirected graph $\bar G$ with nodes and arcs:
\begin{align*}
\bar V &= \{v \in V: \exists \tilde{v} \in \tV \text{ with } \pi(\tilde{v})=v \text{ and } \bar y_{\tilde{v}}>0\},\\
\bar A &= \{a=(v,w) \in A: \exists \tilde{a}=(\tilde{v}, \tilde{w}) \in \tA \text{ with } \pi(\tilde{v})=v, \pi(\tilde{w})=w, \text{ and } \bar x_{\tilde{a}}>0\}.
\end{align*}
We assign assign to arc $a\in \bar A$ the capacity $\displaystyle\max_{\tilde{a}=(\tilde{v}, \tilde{w}) \in \tA: \atop \pi(\tilde{v})=v, \pi(\tilde{w})=w} \bar x_{\tilde{a}}$, that is, the maximum of the (possibly fractional) $x$-values of all arcs connecting any of the replicas associated with the endpoints of arc $a$.

To look for the  set $\bar S$ associated with a \emph{violated} inequality in the shape of \eqref{mC:connectivity_relax}, we construct a tree of min-cuts for $\bar G$ (with the above mentioned capacities), that we denote by $\bar\T$ and that can be computed using the procedure proposed in \citep{gusfield1990}. From $\bar \T$, for each min-cut, $\emptyset \subsetneqq \bar S \subsetneqq \bar V$, with $\bar S \cap P, \bar S^c \cap P \neq \emptyset$ we compute:
$$
\nu_{\bar S} = \sum_{a \in \delta^+(\bar S) }  \bar x_a + \sum_{w \in \bar S^c \cap U} s_w - 1.
$$

In case $\nu_{\bar S} <0$, the inequality \eqref{mC:ctr12bb} associated with $\bar S$, and any $v\in \bar S \cap P$, and $v' \in \bar S^c \cap P$ is violated and is incorporated to the pool of constraints, in the form \eqref{mC:connectivity_relax}.

If $\nu_{\bar S} \geqslant 0$ for all considered min-cuts, the solution does not violate any of  the relaxed connectivity constraints for sets $S$ that either contain all replicas of a pick-up point or do not contain any of them.

\item {\bf Phase II: } The constraints added in the previous phase might be satisfied, however, the connectivity of the resulting network is not guaranteed, since the subgraph is constructed in the expanded graph rather than in the original graph $G$. Therefore, if the first phase does not result in adding any new cut to the pool of constraints, we proceed to properly separate the constraints in \eqref{mC:ctr12bb}. This strategy helps to avoid adding dense constraints to the pool, as this phase of the procedure, according to our experiments, needs to be called only a relatively small number of times.

In this phase, we construct the undirected graph $\hat{G}$ with nodes and arcs:
$$
\hat{V} = \{v \in \tV:  \bar{y}_{v}>0\}, \quad \hat{A}  = \{a \in \tA:  \bar x_{a}>0\}.
$$
Now, we assign assign to arc $a\in \hat A$ the capacity $\bar x_a$.

As in Phase I, we construct a tree of min-cuts for $\hat G$, $\hat \T$. From $\hat \T$,  for each min-cut, $\emptyset \subsetneqq \hat S \subsetneqq \hat V$, we select $v\in \hat S\cap \tP$, $v'\in \hat S^c \cap \tP$ maximizing the value:
$$
\mu_{\hat S, v, v'} = \sum_{a \in \delta^+(\hat S)}  \bar x_a - \bar y_v - \bar y_{v'}+ 1 + \sum_{w \in \hat S^c \cap U} s_w
$$
Then,  if $\mu_{\hat S, v, v'} <0$, the inequality \eqref{mC:ctr12bb} associated with $\hat S$, $v$, and $v'$ is violated and is incorporated to the pool of constraints. Otherwise, the solution does not violate the connectivity constraint.
\end{itemize}

\section{Computational Study}\label{sec:comput}

In this section we report the results of the computational experience that we conducted to validate our proposal and determine the impact of the different parameters in our approaches.

We have generated several instances following a structure similar to the benchmark instances for pick-up and delivery problems~\citep{mosheiov1994travelling}. Let $n=|P|$  be the number of pick-up points and $m=|U|$ the number of drop-off points. For each $n \in \{5, 10, 20, 50\}$  and $m \in\{5, 10, 20\}$ (with $m\leq n$)  we generate $2$-dimensional coordinates for the pick-up and drop-off points in $[-500,500]\times[-500,500]$, and we compute  the cost for traversing each arc linking these points, $C_a$, as the $\ell_1$-distance between the two end-nodes of the arc. In our computational study we always assume that drop-off points have the same set-up cost and therefore, in our model Constraints  \eqref{mC:ctr1} are replaced by Constraints  \eqref{mC:ctr2}. Then, we  consider $p\in\{2, 5, 10, 20\}$ (with $p\leq m$).

For each of the pick-up points, we generate random integer demands, $p_v$, in $[10,100]$. We consider three different capacities for the vehicle, which depend on the  instance, namely $\rho \in \{ 0.25\times (\max_{v\in P} p_v +\sum_{v\in P}p_v),\; 0.5\times (\max_{v\in P} p_v+\sum_{v\in P}p_v),\; \sum_{v\in V}p_v\} $, that we called \texttt{small}, \texttt{medium}, and \texttt{large} capacity, respectively, and we compute the transportation cost per unit flow associated to each arc, $a\in A$, as $C_a'=\frac{C_{a}}{\rho}$. Three different values of $\alpha$ were also considered: $0,\; 0.5$, and $1$,  and the value of $\eta$ is set to $1$ (to assure at least one unit of demand received at each drop-off point to activate it). 

For each of these instances, we solved seven different variants of the model: $(1,1)$-{\rm DOBC}\_C, $(\infty,k)$-{\rm DOBC}\_C, and $(\infty,k)$-{\rm DOBC}\_P, $k\in \{1,2,3\}$. 

In Table \ref{tab:summary_experiments} we summarize the parameters used in our experiments.
\begin{table}[h]
\begin{center}\begin{tabular}{rl}
Parameter & Values\\\hline
$n = |P|$ & $\{5,10,20,50\}$\\
$m= |U|$ & $\{5,10,20\}$ ($m\leq n$)\\
$C_a$ & $\ell_1$-distance between random points in $[-500,500]\times [-500,500]$\\
$C_a'$ & $\frac{C_{a}}{\rho}$\\
$p$ & $\{2,5,10,20\}$\\
$p_v$ & Random integers in $[10,100]$\\
$\rho$ & \texttt{small}: $0.25\times (\max_{v\in P} p_v +\sum_{v\in P}p_v)$\\
& \texttt{medium}: $0.5\times (\max_{v\in P} p_v+\sum_{v\in P}p_v)$\\
& \texttt{large}: $ \sum_{v\in V}p_v\}$\\
$\alpha$ & $\{0, 0.5, 1\}$\\
$\eta$ & 1\\
$(k_d,k_p)$ & $\{(1,1), (\infty,1), (\infty,2), (\infty,3)$\\
T & $\{C,P\}$\\\hline
\end{tabular}
\caption{Parameters and models considered in our experiments.\label{tab:summary_experiments}}
\end{center}
\end{table}
All instances were solved with the Gurobi 11.0.2 optimizer, under a Windows 10 environment on an 11th Gen Intel(R)  Core(TM) i7-11700K @ 3.60 GHz 3.60 GHz processor and 64 GB of RAM. Default values were used for the parameters of Gurobi solver except for TimeLimit, MIPGap, Threads, Heuristics, Cuts and, MIPFocus fixed to 7200 seconds, 0.005, 1, 0, 0 and, 2, respectively.

\subsection{Results for the small-size instances}

In order to determine the impact of the different parameters in our approaches we start the computational study by considering small-size instances where $n \in \{5, 10, 20\}$ with $m \in \{5, 10\}$ ($m \leq n$) and $p\in\{2, 5, 10\}$ ($p\leq m$). For each combination of $(n,\, m,\, p)$ we consider five different instances and, as mentioned before, for each one   we consider three values for parameters $\rho$ and $\alpha$  and we solve the seven variants of the model.  Thus, we have solved a total of 3780 small-size instances. In Table \ref{tabla:1} and Figures \ref{fig:8}--\ref{fig:3} we summarize the average results for instances with the same size and characteristics. 

First, in Table \ref{tabla:1}, we report the average results for these instances, aggregating by the values of $n$, $m$, $p$, and for the different types of formulations. We indicate in this table the percent number of infeasible instances ({\bf Inf}), percent number of instances that reached the time limit of two hours ({\bf TL}), and, in that case the average MIP Gap after this time ({\bf GAP}). As we can see, infeasible instances only appear when using models $(1,1)$-DOBC\_C and $(\infty,1)$-DOBC\_T, both for T=P and T=C, and the number of these instances is considerably higher for model $(1,1)$-DOBC\_C. These are the more restrictive models where a single visit to both pick-up  and drop-off points  for model $(1,1)$-DOBC\_C  and a single visit to pick-up points for models $(\infty,1)$-DOBC\_T  is allowed. Infeasibility no longer happens when at least two visits to pick-up points and as many visits as needed to drop-off points are allowed (models $(\infty,k)$-DOBC\_T, both for T=P and T=C, with $k\geq 2$). Conversely,  the number of instances reaching the time limit increases with the number of visits allowed to the pick-up points. This is due to the increase in the size of the problem as a result of the number of replicas/allowed visits. The MIP Gap in these cases varies between $2\%$ and $7\%$. As can be seen, there is also a significant number of instances that reach the time limit when using the more restrictive model where only one visit to both, the pick-up and the drop-off points is allowed (model $(1,1)$-DOBC\_C) and in this case the MIP Gap reaches $36\%$.

\begin{table}
\centering
{\small
\begin{tabular}{| p{0.15cm}  p{0.15cm}  p{0.2cm} | c | c | c | c | c | c | c | c | c | c | c | c | c |}
\hline
 &  &  & \multicolumn{3}{c}{} &  \multicolumn{10}{|c|}{$(\infty,k)$-{\rm DOBC}\_T}\\
 \cline{7-16}
  &  &  &  \multicolumn{3}{c}{$(1,1)$-{\rm DOBC}\_C}  & \multicolumn{2}{|c|}{$k=1$} &  \multicolumn{4}{c|}{$k=2$} &  \multicolumn{4}{c|}{$k=3$}\\
  \cline{7-16}
    &  &  &  \multicolumn{3}{c|}{} & T=P & T=C &  \multicolumn{2}{c|}{T=P} &   \multicolumn{2}{c|}{T=C} &  \multicolumn{2}{c|}{T=P} &   \multicolumn{2}{c|}{T=C} \\
\hline
$n$ & $m$ & $p$ & \texttt{Inf} & \texttt{TL} & \texttt{GAP} & \texttt{Inf} & \texttt{Inf} & \texttt{TL} & \texttt{GAP} & \texttt{TL} & \texttt{GAP} & \texttt{TL} & \texttt{GAP} & \texttt{TL} & \texttt{GAP} \\
\hline
5 & 5& 2 & 33.33 & - & - & 6.67 & 6.67 & - & - & - &-  & - & - & - & - \\
 &  & 5 & 6.67 & - & - & 6.67 & 6.67 & - & -  & - & - & - & - & - & - \\
\hline
10& 5 & 2 & 33.33 & - & - & - & - & - & - & - & - & - & - & 2.22 & 1.88 \\
 &  & 5 & - & - & - & - & - & - & - & - & - & - & - &  -& - \\
\cline{2-16}
 & 10 & 2 & 33.33 & - & - & - & - & - & - & - & - & - &-  & - & - \\
 &  & 5 & - &-  & - & - & - &  -& - & - & - &-  & - & - & - \\
 &  & 10 & - & - & - & - & - & - & - & - & - &-  & - &-  & - \\
\hline
20 & 5 & 2 & 33.33 & 8.89 & 11.00 & - & - & 4.44 & 3.36 & 4.44 & 3.39 & 26.67 & 6.84 & 20.00 & 4.56 \\
 &  & 5 & - & - & - & - & - & - & - & - & - & 4.44 & 3.09 & - &-  \\
\cline{2-16}
 & 10 & 2 & 33.33 & 20.00 & 36.34 & - &  -& 15.56 & 3.80 & 2.22 & 6.60 & 33.33 & 7.08 & 26.67 & 4.33 \\
 &  & 5 & - & - & - &  -& - & - & - & - & - & 8.89 & 3.47 & 4.44 & 4.09 \\
 & & 10 & - & - & - &  -& - & - & - & - & - & 6.67 & 2.07 & 2.22 & 3.30 \\
\hline\hline
\multicolumn{3}{|c|}{\textbf{Total}} & {14.44} & {2.41} & - & {1.11} & {1.11} & {1.67} & - & {0.56} & - & {6.67} & {5.96} & {4.63} & {4.26} \\
\hline

\end{tabular}}

\caption{Percentage of infeasible instances, and percentage of instances that reached the time limit and their MIP Gap.\label{tabla:1}}
\end{table}

Regarding the instances that are feasible for all compared models, first, we note that the size of the instance ($n$) has a significative impact in the computational requirement of the approaches, as can be seen both in the number of instances that reached the time limit and in the performance profiles  that we draw in Figure \ref{fig:8}. Comparing the more restrictive model,  $(1,1)$-DOBC\_C, with model $(\infty,1)$-DOBC\_C,  where one visit is allowed for the pick-up points, but multiple visits are allowed for the drop-off points, we have observed that the first one is computationally more difficult, both in the number of instances that overcome the maximum MIPGap allowed of $0.5\%$ after the time limit, and in the CPU times required to solve the feasible instances (for all those represented in the plot), as can be seen in the performance profiles (with CPU time in logarithmic scale for ease of comparison) drawn in Figure \ref{fig:1} (left).

\begin{figure}[]
\begin{center}
\includegraphics[width=0.5\textwidth]{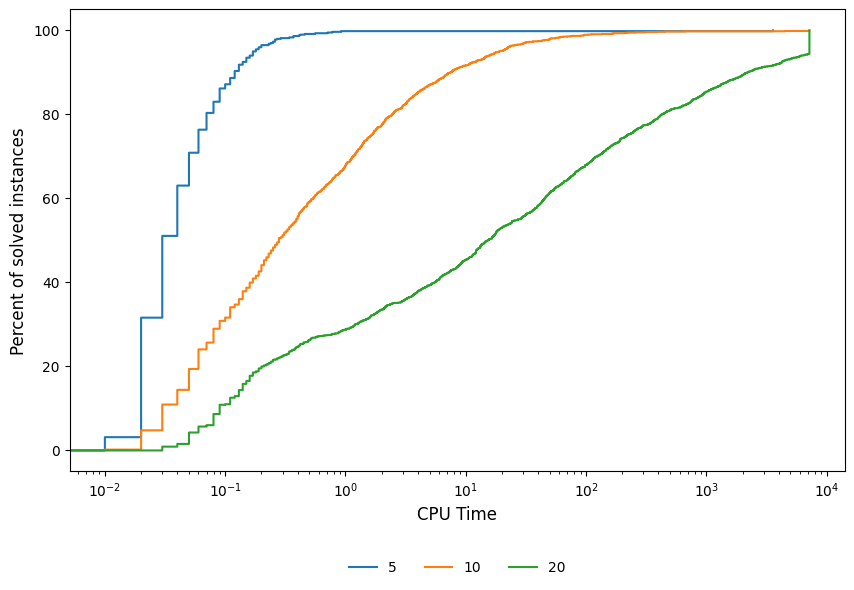}
\caption{Performance profile by the instance size ($n$),  with CPU time in log scale}. \label{fig:8}
\end{center}
\end{figure}

\begin{figure}[h]
\begin{center}
\includegraphics[width=0.5\textwidth]{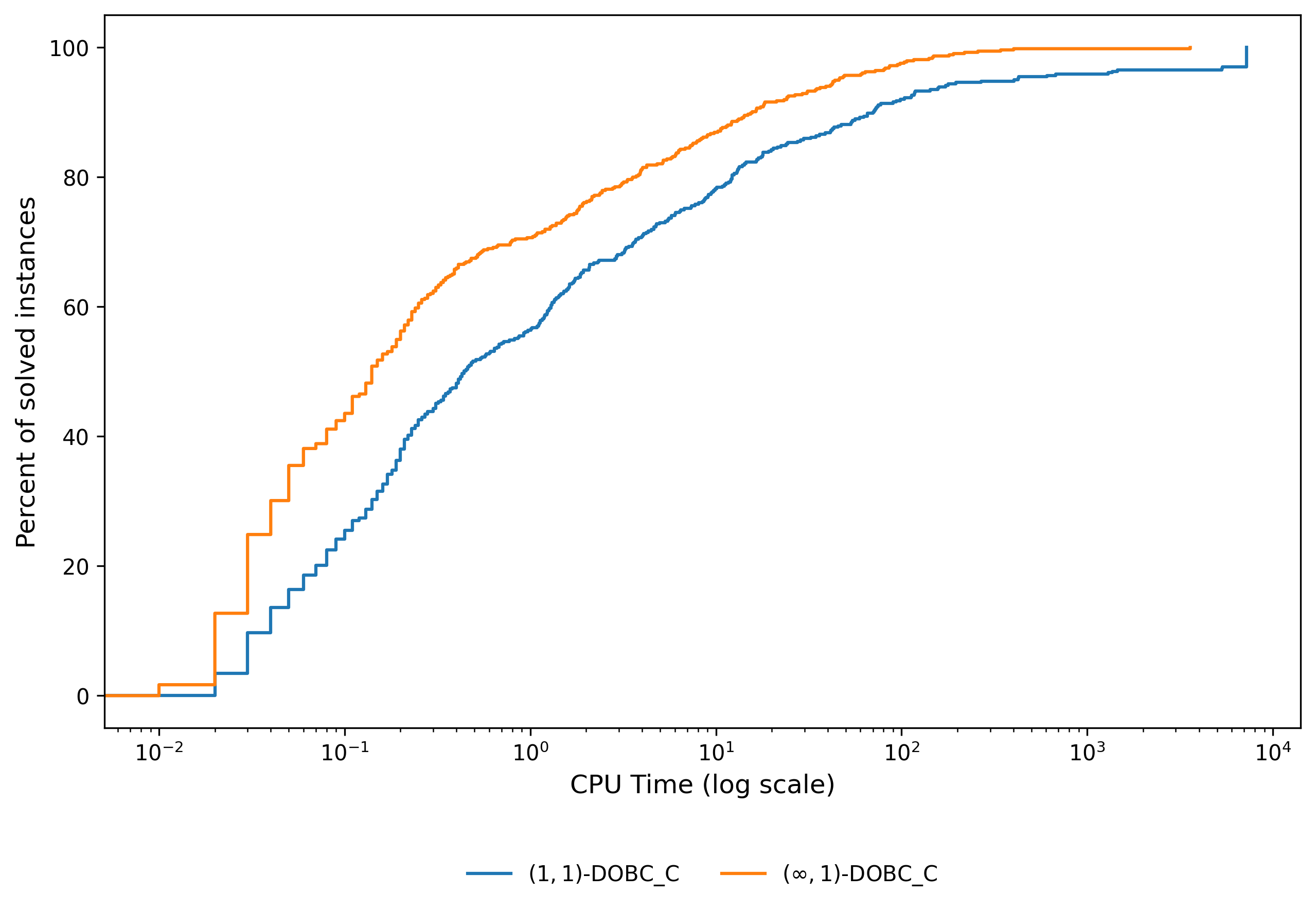}~\includegraphics[width=0.5\textwidth]{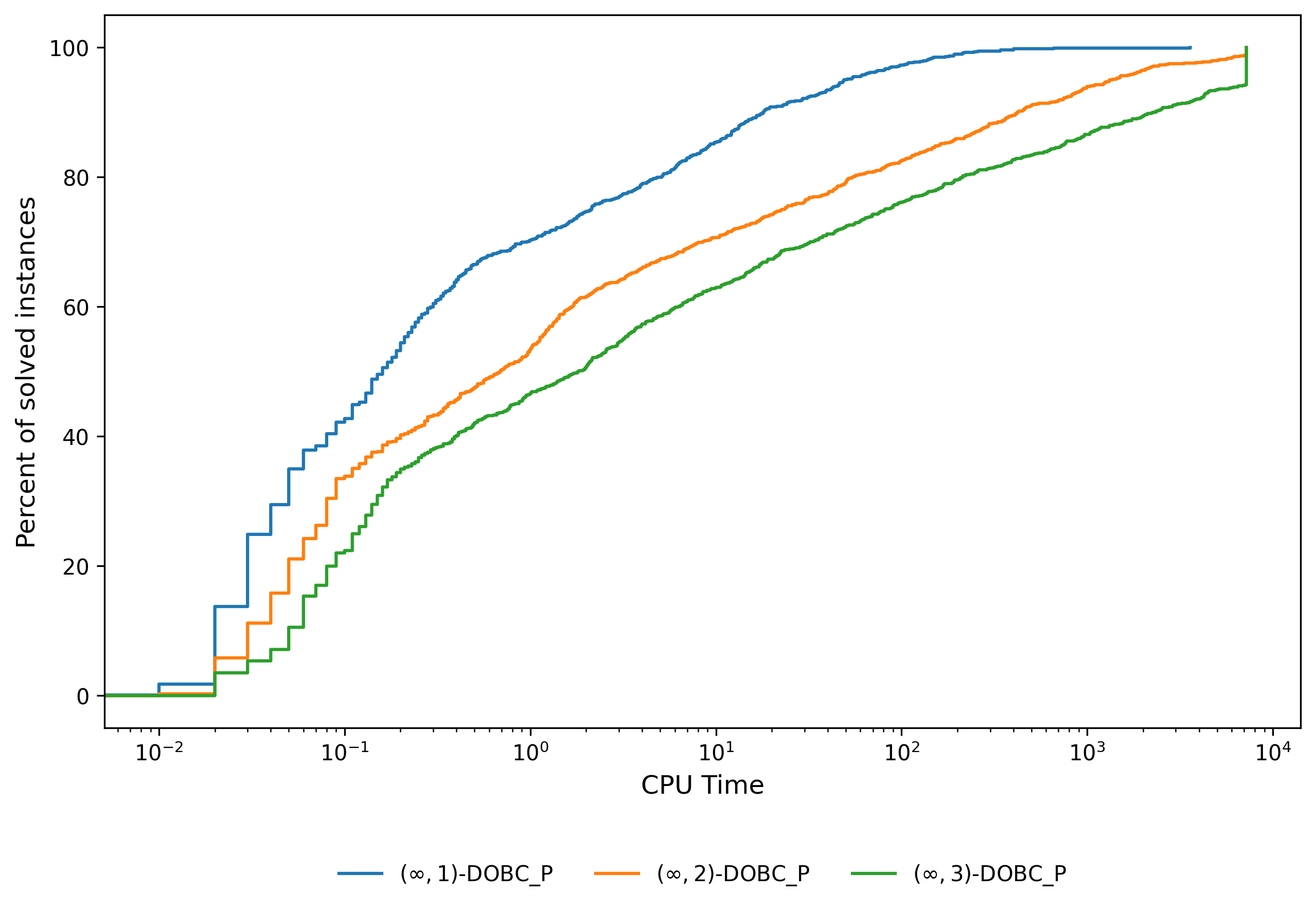}
\caption{Performance profiles of the $(1,1)$ -DOBC\_C and $(\infty,1)$-DOBC\_C models for feasible instances (left) and { $(\infty,k)$-DOBC\_P} for $k \in \{1,2,3\}$ (right).  CPU time in log scale.\label{fig:1}}
\end{center}
\end{figure}

\begin{figure}[h]
\begin{center}
\includegraphics[width=0.5\textwidth]{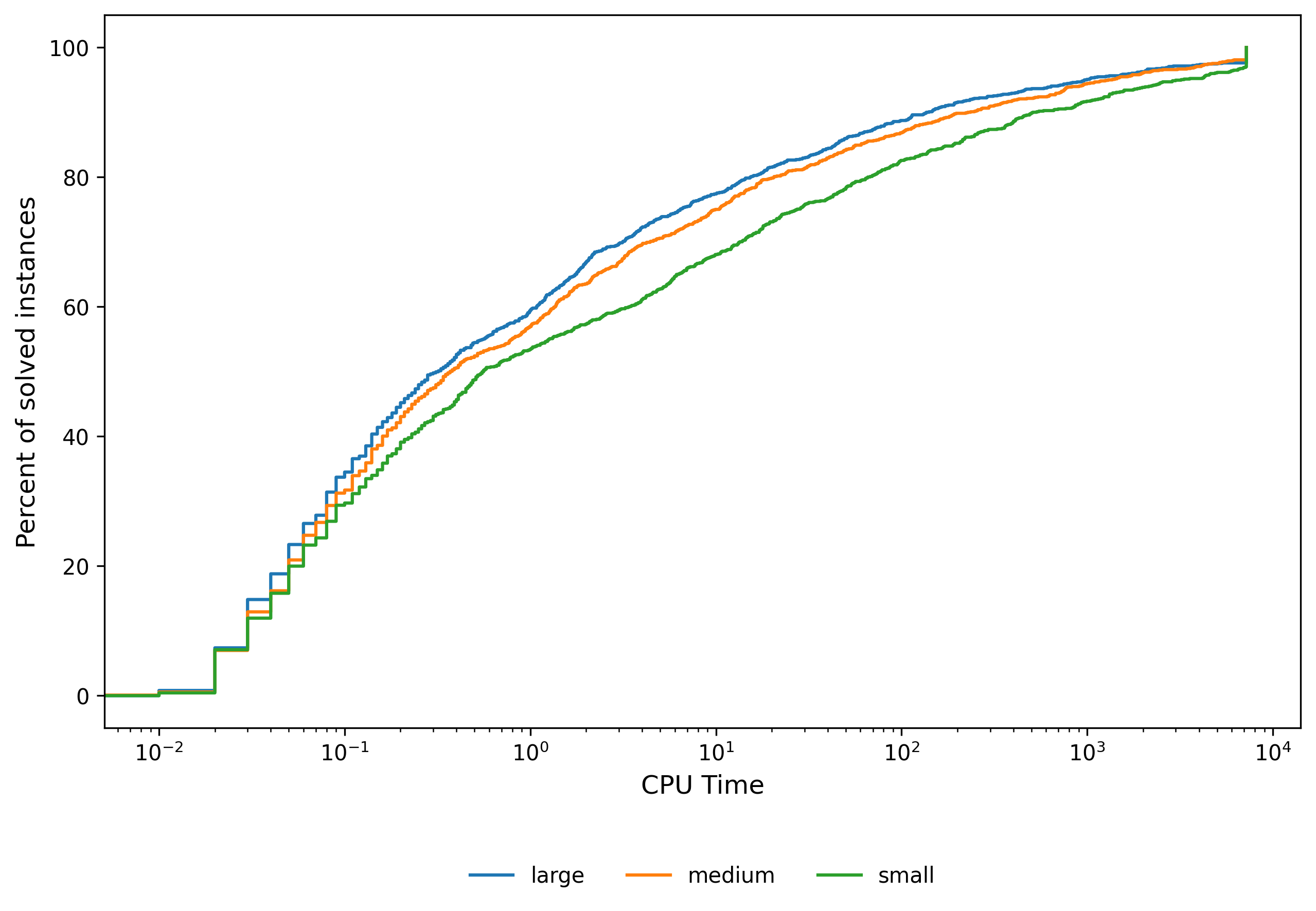}
\caption{Performance profile by the vehicles capacities ($\rho$ values).  CPU time in log scale.\label{fig:5}}
\end{center}
\end{figure}

Comparing models $(\infty,k)$-DOBC\_P, we can observe that the larger is the value of $k$ the more  challenging are problems, and the  time required to solve the instances, as well as the number of instances that reached the time limit increases considerably. This can be seen in Figure \ref{fig:1} (right), which  compares the time required to solve the feasible instances depending on the maximum number of allowed visits to the pick-up points (three different values of $k$).

\begin{figure}[h]
\begin{center}
\includegraphics[width=0.5\textwidth]{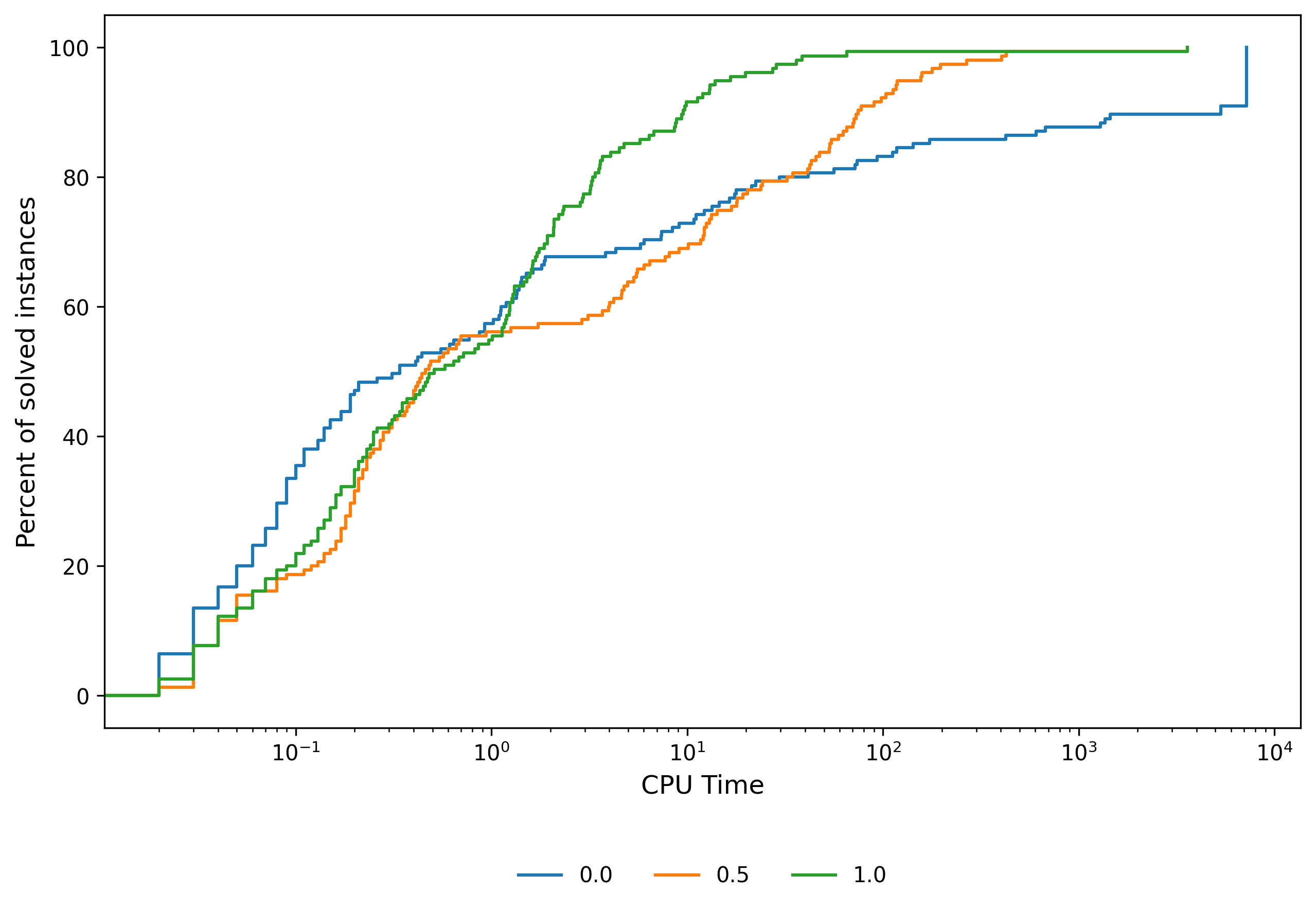}~\includegraphics[width=0.5\textwidth]{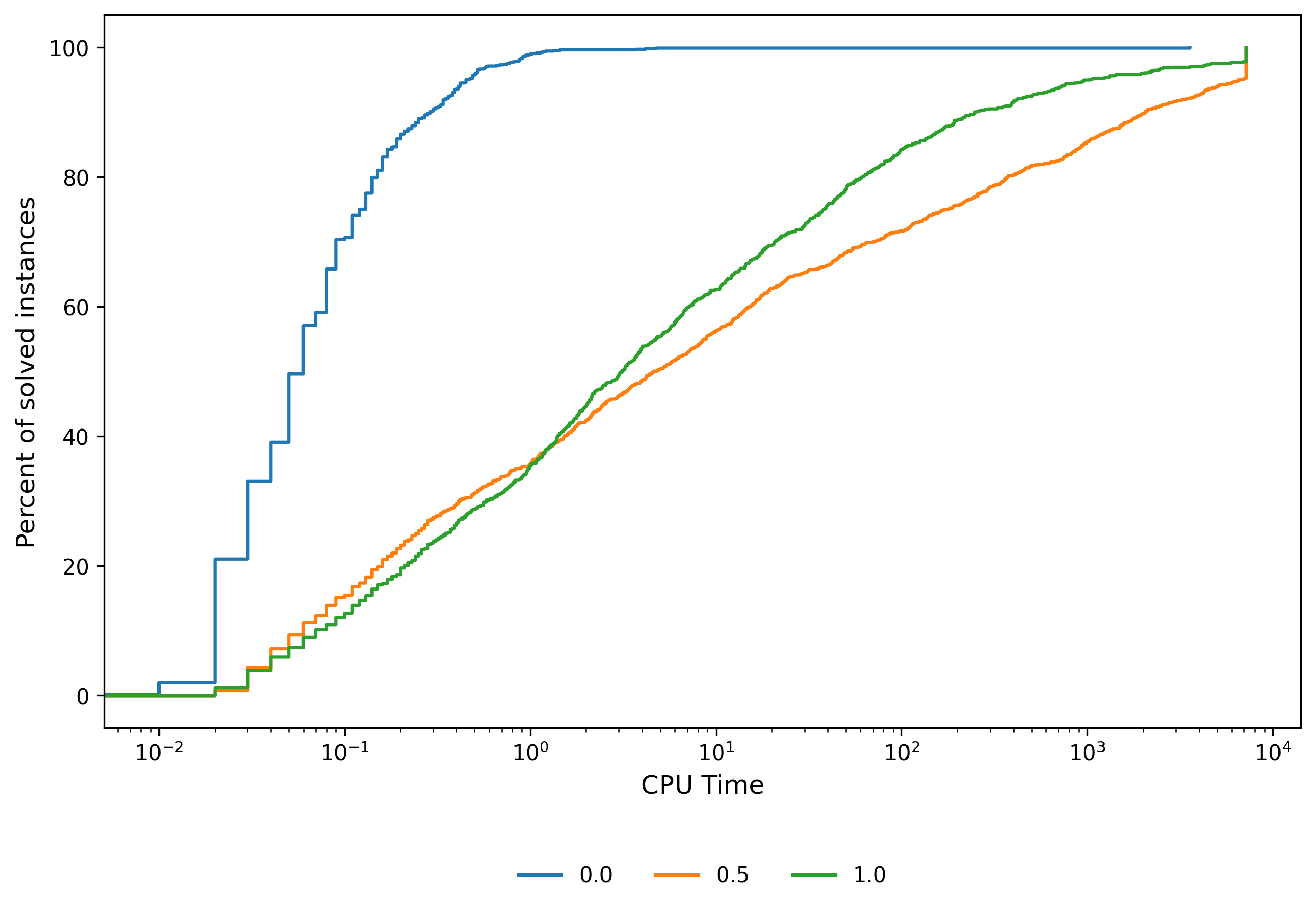}
\caption{Performance profile by value of $\alpha$ (left:   $(1,1)$ -DOBC\_C, right: $(\infty,k)$-DOBC\_P} for all $k$).   CPU time in log scale.\label{fig:4}
\end{center}
\end{figure}

In Figures \ref{fig:5} and  \ref{fig:4}   we analyze the computational implications of the different parameters that we consider in our experiments, namely $\rho, \alpha$. In Figure \ref{fig:5} we draw the performance profile for the three types of $\rho$-values (capacities of the vehicles). Note that the capacity of the vehicle may require more than one visit to pick up all the demand and therefore, the most time-demanding instances are those with smaller capacities, although the differences are not too significant.

On the other hand, the value of $\alpha$ has also an impact in the computational performance of the models. In Figure \ref{fig:4} we plot the performance profile for both $(1,1)$-DOBC\_C (left) and $(\infty,k)$-DOBC\_P (right), considering together all the values of  $k$. As one can observe,  instances with smaller $\alpha$ values seem to be easier to solve, especially for the $(\infty,k)$-DOBC\_P instances, where the differences are evident.

We have also analyzed the computational implications  of the requirement of cycle solutions or not and we have checked that  there are not significant differences between the models that require cycle solutions and those that allow the solution to be a path.

\subsection{Large-size instances}

With this second block of experiments we aim to find out the size of the instances that we are able to solve up to optimality with our models. In Table \ref{tabla:2} and Figure \ref{fig:9} we summarize the average results for the instances with size $n=50$,  $m \in \{5, 10, 20\}$ and $p\in\{2, 5, 10, 20\}$ with $p\leq m$. Although we also run the $n=100$ instances, all the approaches reached the time limit, with large MIPGaps, so the results are not reported. Furthermore, model $(1,1)$-DOBC\_C was not able to solve to optimality of the instances,  so we only report the results for $(\infty,k)$-DOBC\_P with $k \in \{1,2,3\}$. 

In Table \ref{tabla:2} we summarize the results for $(\infty,k)$-DOBC\_P with $k \in \{1,2,3\}$ in terms of the number of  instances that overcame the MIPGap limit of $0.5\%$ after the time limit, out a total of 30 instances of the same size (\textbf{UnS})   and the average MIP Gaps (\textbf{Gap}) obtained in these cases. In this table, we show only the results for $\alpha=1$, since   for $\alpha=0$ all the instances were optimally solved, and for $\alpha=0.5$ none of the instances was optimally solved. In  Figure \ref{fig:9} we compare the computational time required to optimally solve the instances for $\alpha=0$ (focusing on flow costs) and $\alpha=1$ (focusing on link costs). In this figure, where the times are dissagregated by the value of $k$, we can observe that instances are easier to solve for $\alpha=0$  than for $\alpha=1$.

Table \ref{tabla:2} and Figure \ref{fig:9} show that, as expected, both the computational time required to solve the problems and the MIP Gap increase with the value of $k$.

Finally, we have checked that the results for the large instances are not affected by the value of $\rho$ nor by the requirement of being a cycle or not in the models. 

\begin{table}
\centering
{\small\begin{tabular}{| c  c  c | c | c | c | c | c | c |}
\hline
\multicolumn{9}{|c|}{$(\infty,k)$-DOBC\_P}\\
\hline
\multicolumn{3}{|c}{} & \multicolumn{2}{|c}{$k=1$} & \multicolumn{2}{|c}{$k=2$} & \multicolumn{2}{|c|}{$k=3$}\\\hline
\textbf{n} & \textbf{m} & \textbf{p} & \texttt{Gap}(\%) & \textbf{UnS} & \texttt{Gap}(\%) & \textbf{UnS} & \texttt{Gap}(\%) & \textbf{UnS} \\
\hline
\multirow{9}{*}{$50$} & \multirow{2}{*}{$5$} & \textbf{2} & 6.42 & 18 & 20.21 & 30 & 40.19 & 30 \\
 &  & \textbf{5} & 3.26 & 12 & 12.64 & 30 & 33.02 & 30 \\
\cline{2-9}
 & \multirow{3}{*}{$10$} & \textbf{2} & 8.81 & 27 & 20.78 & 30 & 42.71 & 30 \\
 &  & \textbf{5} & 1.48 & 12 & 10.87 & 30 & 32.25 & 30 \\
 &  & \textbf{10} & 1.43 & 9 & 9.19 & 29 & 29.58 & 30 \\
\cline{2-9}
 & \multirow{4}{*}{$20$} & \textbf{2} & 11.44 & 25 & 26.14 & 30 & 49.77 & 30 \\
 &  & \textbf{5} & 2.29 & 6 & 10.56 & 23 & 35.76 & 30 \\
 &  & \textbf{10} & 0.49 & 0 & 1.39 & 13 & 15.94 & 29 \\
 &  & \textbf{20} & 0.49 & 0 & 1.42 & 11 & 15.09 & 29 \\
\hline
\multicolumn{3}{|c|}{\textbf{Total}}  & \textbf{4.01} & \textbf{109} & \textbf{12.58} & \textbf{226} & \textbf{32.70} & \textbf{268} \\
\hline

\end{tabular}}
\caption{Number of unsolved up to optimality instances (out of 30 per row) and  average percentage MIP Gap.\label{tabla:2}}
\end{table}

\begin{figure}[h]
\begin{center}
\includegraphics[width=0.5\textwidth]{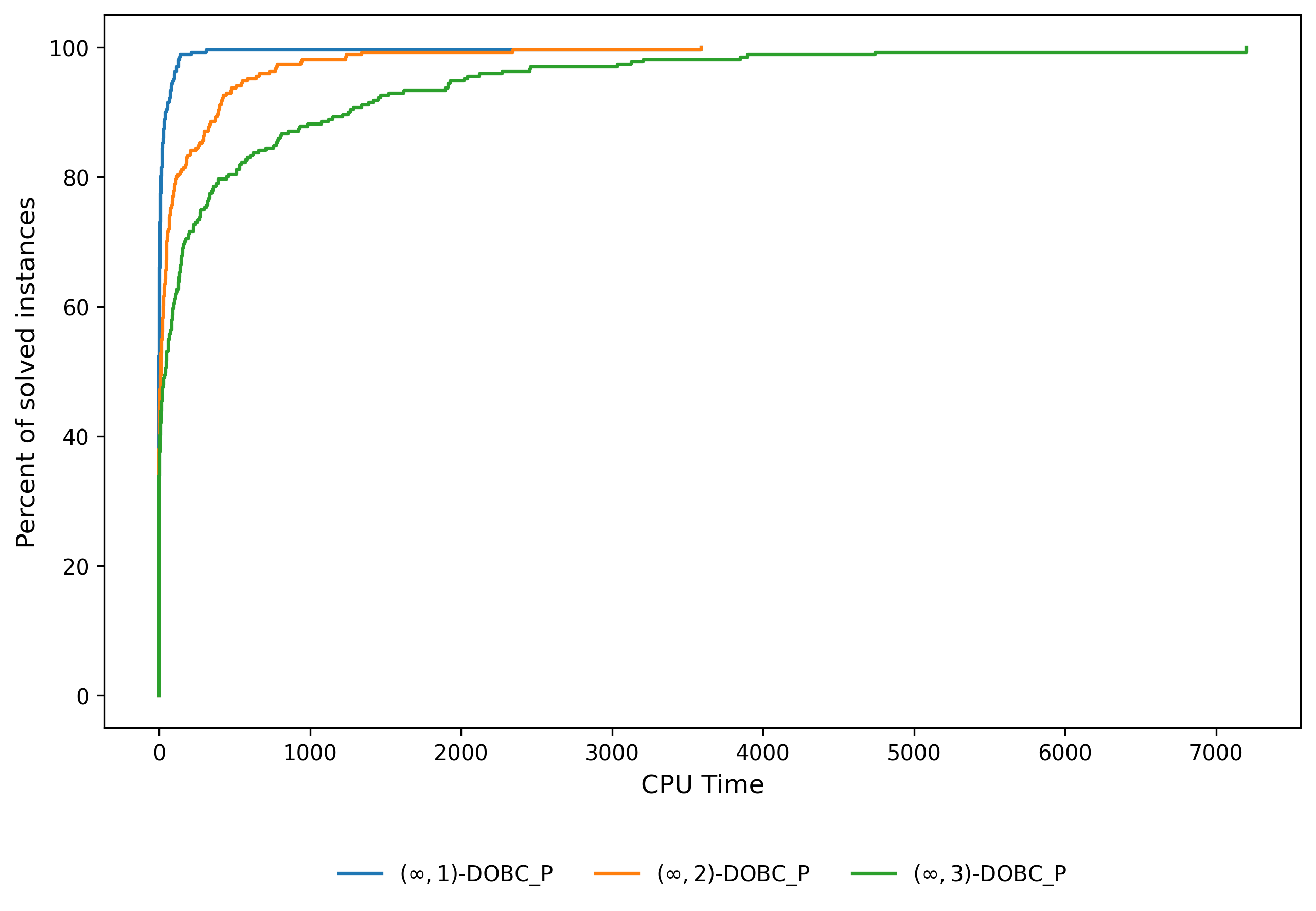}~\includegraphics[width=0.5\textwidth]{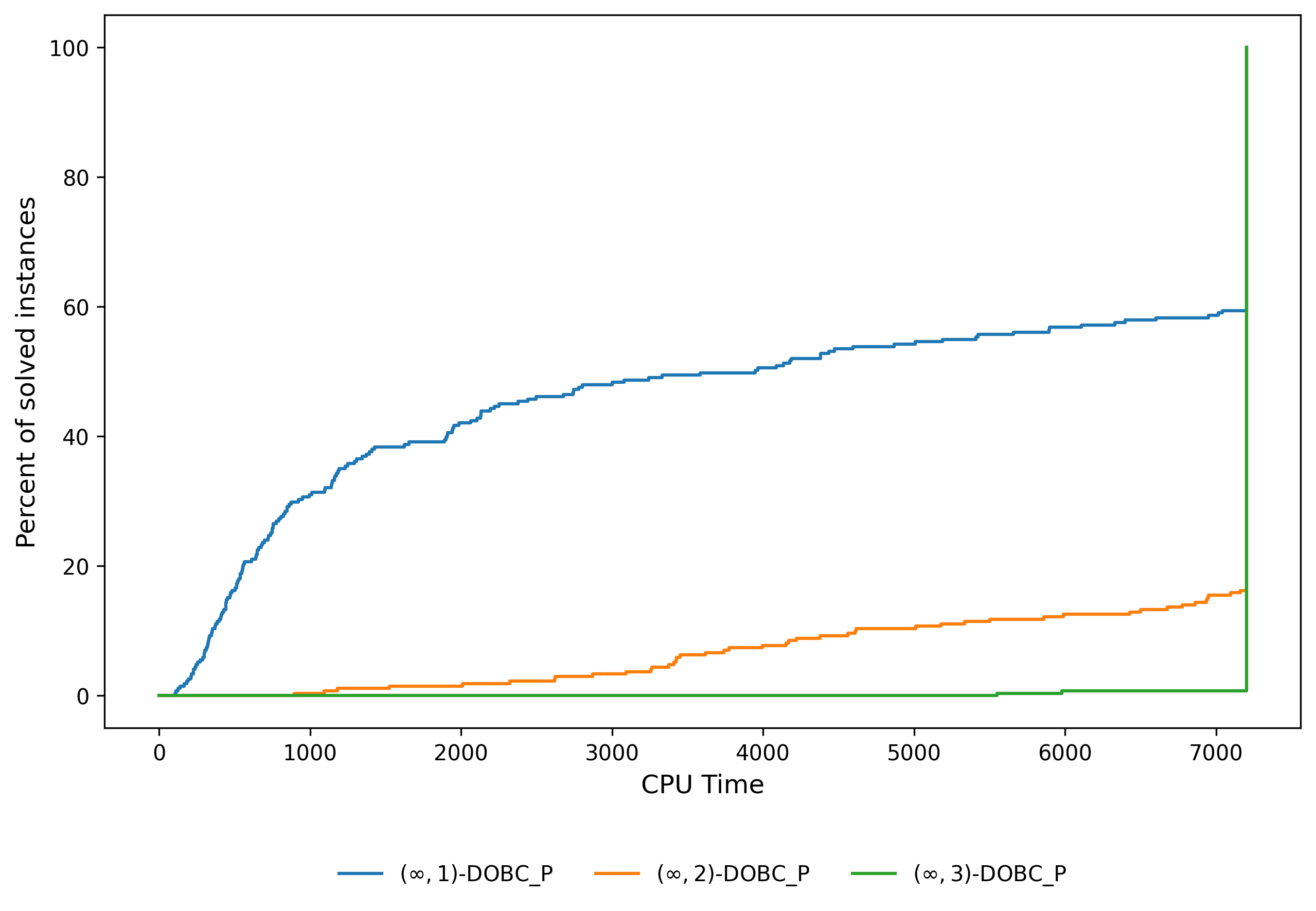}
\caption{Performance profile by the value of $k$ for $\alpha=0$ (left) and $\alpha=1$ (right). \label{fig:9}}
\end{center}
\end{figure}

\subsection{Infeasibility, number of replicas, and objective value}

\begin{table}[h!]
\centering
{\small
\begin{tabular}{| l | p{0.1cm} | p{0.3cm} | p{0.3cm} | p{0.3cm} | p{0.3cm} || p{0.1cm} | p{0.3cm} | p{0.3cm} | p{0.3cm} | p{0.3cm}  || p{0.1cm} | p{0.3cm} | p{0.3cm} | p{0.3cm} | p{0.3cm}  || p{0.1cm} | p{0.3cm} | p{0.3cm} | p{0.3cm} | p{0.3cm}  |}
\hline
\multicolumn{21}{c}{$\rho = \texttt{quantile}$}\\\hline
 & \multicolumn{5}{|c||}{5\%}  & \multicolumn{5}{|c||}{10\%}   &   \multicolumn{5}{|c||}{15\%}  & \multicolumn{5}{|c|}{20\%}  \\
\hline
$n\backslash k$ & $1$ & $2$ & $3$ & $4$ & $5$ & $1$ & $2$ & $3$ & $4$ & $5$  & $1$ & $2$ & $3$ & $4$ & $5$ &  $1$ & $2$ & $3$ & $4$ & $5$ \\
\hline\hline
$5$  & 1 & 1 & 0.6 & 0.2 & 0.2 & 1 & 1 & 0.4 & 0.2 & 0.2 & 1 & 1 & 0.4 & 0.2 & 0.2 & 1 & 0.8 & 0.2 & 0.2 & 0.2 \\\hline
$10$ & 1 & 1 & 0.8 & 0.6 & 0.6 & 1 & 1 & 0.8 & 0.6 & 0.2 & 1 & 1 & 0.8 & 0.2 & 0.2 & 1 & 1 & 0.6 & 0.2 & 0 \\\hline
$20$& 1 & 1 & 1 & 1 & 0.2 & 1 & 1 & 0.6 & 0.2 & 0.2 & 1 & 1 & 0.4 & 0.2 & 0.2 & 1 & 1 & 0.2 & 0.2 & 0.2 \\\hline
$50$ & 1 & 1 & 1 & 1 & 1 & 1 & 1 & 1 & 1 & 0.6 & 1 & 1 & 1 & 0.6 & 0.2 & 1 & 1 & 0.6 & 0.4 & 0 \\\hline
$100$ & 1 & 1 & 1 & 1 & 1 & 1 & 1 & 1 & 1 & 0.4 & 1 & 1 & 1 & 0.6 & 0 & 1 & 1 & 1 & 0 & 0 \\\hline
\multicolumn{21}{c}{$\rho = \texttt{quantile}$}\\\hline
& \multicolumn{5}{|c||}{25\%}  & \multicolumn{5}{|c||}{40\%}   &   \multicolumn{5}{|c||}{55\%}  & \multicolumn{5}{|c|}{75\%}  \\
\hline\hline
$n\backslash k$ & $1$ & $2$ & $3$ & $4$ & $5$ & $1$ & $2$ & $3$ & $4$ & $5$ &  $1$ & $2$ & $3$ & $4$ & $5$ &  $1$ & $2$ & $3$ & $4$ & $5$ \\\hline\hline
$5$ &1 & 0.8 & 0.2 & 0.2 & 0 & 1 & 0.6 & 0.2 & 0 & 0 & 1 & 0.2 & 0 & 0 & 0 & 1 & 0 & 0 & 0 & 0 \\\hline
$10$ &1 & 0.8 & 0.4 & 0.2 & 0 & 1 & 0.6 & 0.2 & 0 & 0 & 1 & 0 & 0 & 0 & 0 & 1 & 0 & 0 & 0 & 0 \\\hline
$20$ &1 & 0.8 & 0.2 & 0 & 0 & 1 & 0.2 & 0 & 0 & 0 & 1 & 0 & 0 & 0 & 0 & 1 & 0 & 0 & 0 & 0 \\\hline
$50$ &1 & 1 & 0.4 & 0 & 0 & 1 & 0.8 & 0 & 0 & 0 & 1 & 0 & 0 & 0 & 0 & 1 & 0 & 0 & 0 & 0 \\\hline
$100$&1 & 1 & 0.4 & 0 & 0 & 1 & 1 & 0 & 0 & 0 & 1 & 0 & 0 & 0 & 0 & 1 & 0 & 0 & 0 & 0 \\\hline
\end{tabular}}
\caption{Proportion of infeasible instances.\label{table:3}}
\end{table}

We finish our computational study by analyzing the convenience of bounding above the number of visits to the pickup points in the models (taking into account that the larger the value of $k$ in models $(\infty,k)$-DOBC\_T the larger are the computation times) based on both the infeasibility of the problem and the optimal objective values of the feasible instances.

First, observe that, given an instance,  one can compute the minimum number of pick-up points' replicas, $k$, required to assure feasibility, in terms of the demands of the pick-up points, $p_v$ for all $v\in P$, and the capacity of the vehicle, $\rho$. Note that the instance is feasible if and only if:
$$
\frac{\max_{v \in P} p_v}{\rho} \leq k.
$$
Then, in Table \ref{table:3} we analyze the proportion of instances that are infeasible for the different values of $k \in \{1,2,3,4,5\}$, when defining $\rho$ as a different quantile of the demands $\{p_v\}_{v\in P}$. Observe that with this in mind, one could select the adequate number of replicas, $k$, required to assure feasibility of the problem depending on the capacity of the vehicle $\rho$. Although one could choose a larger value of $k$, that would be also feasible, the values of the objective values do no improve significantly (see Figure \ref{fig:3}), while the computational times increase considerably.

\begin{figure}[h]
\begin{center}
\includegraphics[width=0.5\textwidth]{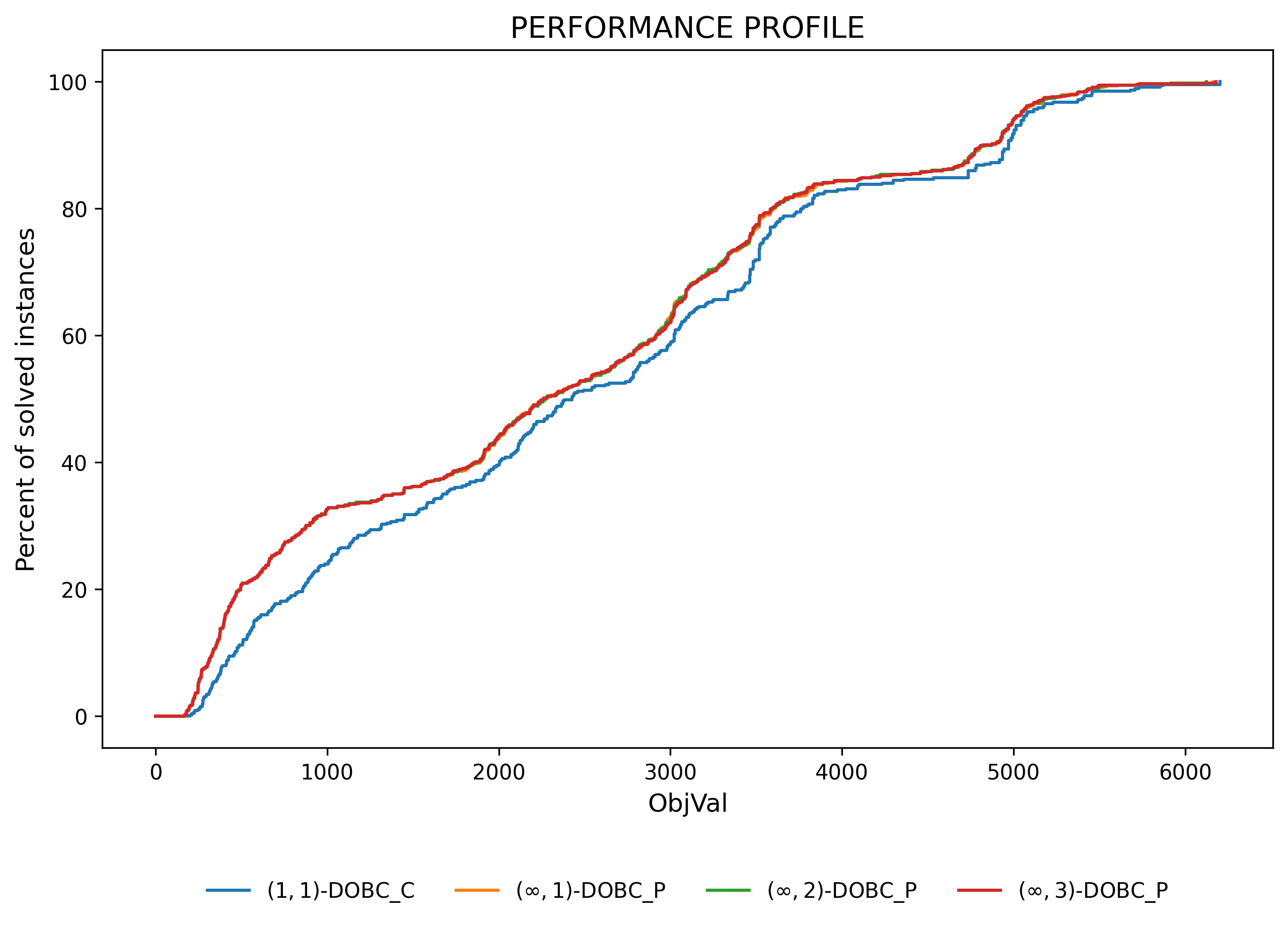}
\caption{Performance profile for the objective value by the different models. \label{fig:3}}
\end{center}
\end{figure}

Observe, the only way that an instance is feasible for $k=1$ is that $\rho \geq \max_{v} p_v$, that is, that the capacity of the vehicle is  greater than the maximum demand. Otherwise, the instance will be infeasible. In other cases, at least $k=2$ is required, that is, more than one visit to the pickup points must be allowed. In case $\rho$ is small enough, for instance, the $5\%$ quantile of the demands, even $5$ visits may result in infeasibility (see rows for $n=50$ or $n=100$ in the table, where all the instances are infeasible for this value of $\rho$). In the other extreme, for $\rho$ being the $55\%$ quantile most of our instances (except $20\%$ of the $n=5$ instances) are feasible for $k=2$. Both, this table and the formula above that relates the capacity of the vehicle, $\rho$, with the maximum demand at the pick-up points, may serve as a tool to decide the adequate number of  visits to the pick-up points,  $k$, that must be allowed in our models.


\section{Conclusions}
\label{sec:conclusions}

In this paper, we introduced the location-routing problem with drop-offs and budget constraints (DOBC), which integrates the selection of facility locations from an affordable set with the design of a vehicle route to collect demand from specified points and drop it off at the selected facilities. We proposed a novel mathematical optimization framework based on an alternative encoding of the underlying graph on which the route is defined. This flexible modeling approach captures a range of relevant variants and enables further analytical developments.

To solve the problem efficiently, we developed a two-phase branch-and-cut algorithm capable of effectively separating the exponential family of constraints required to ensure route connectivity. The proposed methodology was validated through an extensive set of computational experiments, where we examined the behavior of the model under various parameter settings. In particular, we analyzed the role of imposing upper bounds on the number of visits to pickup points, and showed how such bounds can be instrumental in guaranteeing the feasibility of the problem.

Several promising extensions of the proposed problem remain open for future research. First, as commonly encountered in similar real-world applications, demand is often subject to uncertainty, and should therefore be modeled accordingly. There are multiple ways to incorporate this uncertainty into the DOBC framework—such as robust or stochastic optimization approaches, and each alternative deserves further investigation.

Second, a natural and practically relevant extension involves the multi-vehicle case, which introduces an additional layer of complexity. In this setting, not only must facility locations and vehicle routes be determined, but it is also necessary to assign pickup points to specific vehicles, effectively clustering visits to ensure efficient service. These extensions represent fruitful directions for advancing both the modeling and algorithmic contributions of this work.

\section*{Acknowledgements}

The authors acknowledge financial support by  grants PID2020-114594GB-C21 funded by MICIU/AEI/ 10.13039/501100011033; 
PID2022-139219OB-I00 funded by MCIN/AEI/10.13039/501100011033/FEDER-EU;
RED2022-134149-T funded by MICIU/AEI/10.13039/501100011033 (Thematic Network on Location Science and Related Problems); FEDER+Junta de Andalucía projects C‐EXP‐139‐UGR23, and AT 21\_00032; SOL2024-31596
and SOL2024-31708 funded by US;  the IMAG-Maria de Maeztu grant CEX2020-001105-M /AEI /10.13039/501100011033; and the IMUS--Maria de Maeztu grant CEX2024-001517-M.

\end{document}

%% file: graph_ex_instance.tex
\begin{tikzpicture}[scale=1]
\coordinate (A1) at (1,6);
\coordinate (A2) at (4,6);
\coordinate (A3) at (6,7);
\coordinate (A4) at (3,4);
\coordinate (A5) at (5,1);
\coordinate (A6) at (2,2);
\coordinate (B1) at (2,5);
\coordinate (B2) at (6,4);
\coordinate (B3) at (3,1);
\coordinate (B4) at (1,1);
\coordinate (B5) at (5,3);

\node[circle,draw,inner sep=0.5pt](A-1) at (A1) {\tiny $a_1$};
\node[circle,draw,inner sep=0.5pt](A-2) at (A2) {\tiny $a_2$};
\node[circle,draw,inner sep=0.5pt](A-3) at (A3) {\tiny $a_3$};
\node[circle,draw,inner sep=0.5pt](A-4) at (A4) {\tiny $a_4$};
\node[circle,draw,inner sep=0.5pt](A-5) at (A5) {\tiny $a_5$};
\node[circle,draw,inner sep=0.5pt](A-6) at (A6) {\tiny $a_6$};
\node[rectangle,draw, inner sep=1.5pt](B-1) at (B1) {\tiny $b_1$};
\node[rectangle,draw, inner sep=1.5pt](B-2) at (B2) {\tiny $b_2$};
\node[rectangle,draw, inner sep=1.5pt](B-3) at (B3) {\tiny $b_3$};
\node[rectangle,draw, inner sep=1pt](B-4) at (B4) {\tiny $b_4$};
\node[rectangle,draw, inner sep=1pt](B-5) at (B5) {\tiny $b_5$};

\node[above, inner sep=1.5pt, label=90:{\tiny $50$}] at (A-1) {};
\node[above,  inner sep=1.5pt,label=90:{\tiny $50$}] at (A-2) {};
\node[above,  inner sep=1.5pt,label=90:{\tiny $50$}] at (A-3) {};
\node[above,  inner sep=1.5pt,label=90:{\tiny $50$}] at (A-4) {};
\node[above,  inner sep=1.5pt,label=90:{\tiny $50$}] at (A-5) {};
\node[above,  inner sep=1.5pt,label=90:{\tiny $50$}] at (A-6) {};

\end{tikzpicture}

%% file: graph_ex0.tex
\begin{tikzpicture}[scale=1]
\coordinate (A1) at (1,6);
\coordinate (A2) at (4,6);
\coordinate (A3) at (6,7);
\coordinate (A4) at (3,4);
\coordinate (A5) at (5,1);
\coordinate (A6) at (2,2);
\coordinate (B1) at (2,5);
\coordinate (B2) at (6,4);
\coordinate (B3) at (3,1);
\coordinate (B4) at (1,1);
\coordinate (B5) at (5,3);

\node[circle,draw,inner sep=0.5pt](A-1) at (A1) {\tiny $a_1$};
\node[circle,draw,inner sep=0.5pt](A-2) at (A2) {\tiny $a_2$};
\node[circle,draw,inner sep=0.5pt](A-3) at (A3) {\tiny $a_3$};
\node[circle,draw,inner sep=0.5pt](A-4) at (A4) {\tiny $a_4$};
\node[circle,draw,inner sep=0.5pt](A-5) at (A5) {\tiny $a_5$};
\node[circle,draw,inner sep=0.5pt](A-6) at (A6) {\tiny $a_6$};
\node[rectangle,draw, fill=gray!10, inner sep=1.5pt](B-1) at (B1) {\tiny $b_1$};
\node[rectangle,draw, fill=gray!10, inner sep=1.5pt](B-2) at (B2) {\tiny $b_2$};
\node[rectangle,draw, fill=gray!10, inner sep=1.5pt](B-3) at (B3) {\tiny $b_3$};
\node[rectangle,draw, inner sep=1pt](B-4) at (B4) {\tiny $b_4$};
\node[rectangle,draw, inner sep=1pt](B-5) at (B5) {\tiny $b_5$};

\node[above, inner sep=1.5pt, label=90:{\tiny \phantom{50}}] at (A-3) {};

\draw[->, arrows={-Stealth}] (B-1)--(A-1);
\draw[->, arrows={-Stealth}] (A-1)  to[out=270,in=170] (B-1);
\draw[->, arrows={-Stealth}] (B-1)--(A-2);
\draw[->, arrows={-Stealth}] (A-2)--(B-2);
\draw[->, arrows={-Stealth}] (B-2)--(A-3);
\draw[->, arrows={-Stealth}] (A-3)--(A-2);
\draw[->, arrows={-Stealth}] (A-2)--(A-4);
\draw[->, arrows={-Stealth}] (A-4)--(B-2);
\draw[->, arrows={-Stealth}] (B-2)--(A-5);
\draw[->, arrows={-Stealth}] (A-5)--(B-3);
\draw[->, arrows={-Stealth}] (B-3)--(A-6);
\draw[->, arrows={-Stealth}] (A-6)--(A-4);
\draw[->, arrows={-Stealth}] (A-4)--(B-3);
\end{tikzpicture}

%% file: graph_ex1.tex
\begin{tikzpicture}[scale=1.1]
\coordinate (A1) at (1,6);
\coordinate (A2) at (4,6);
\coordinate (A3) at (6,7);
\coordinate (A4) at (3,4);
\coordinate (A5) at (5,1);
\coordinate (A6) at (2,2);
\coordinate (B1) at (2,5);
\coordinate (B2) at (6,4);
\coordinate (B3) at (3,1);
\coordinate (B4) at (1,1);
\coordinate (B5) at (5,3);

\node[circle,draw,inner sep=0.5pt](A-1) at (A1) {\tiny $a_1$};
\node[circle,draw,inner sep=0.5pt](A-2) at (A2) {\tiny $a_2$};
\node[circle,draw,inner sep=0.5pt](A-3) at (A3) {\tiny $a_3$};
\node[circle,draw,inner sep=0.5pt](A-4) at (A4) {\tiny $a_4$};
\node[circle,draw,inner sep=0.5pt](A-5) at (A5) {\tiny $a_5$};
\node[circle,draw,inner sep=0.5pt](A-6) at (A6) {\tiny $a_6$};
\node[rectangle,draw, fill=gray!10, inner sep=1.5pt](B-1) at (B1) {\tiny $b_1$};
\node[rectangle,draw, fill=gray!10, inner sep=1.5pt](B-2) at (B2) {\tiny $b_2$};
\node[rectangle,draw, fill=gray!10, inner sep=1.5pt](B-3) at (B3) {\tiny $b_3$};
\node[rectangle,draw, inner sep=1pt](B-4) at (B4) {\tiny $b_4$};
\node[rectangle,draw, inner sep=1pt](B-5) at (B5) {\tiny $b_5$};

\draw[->, arrows={-Stealth}] (B-1)--(A-1);
\draw[->, arrows={-Stealth}] (A-1)--(A-2);
\draw[->, arrows={-Stealth}] (A-2)--(B-2);
\draw[->, arrows={-Stealth}] (B-2)--(A-3);
\draw[->, arrows={-Stealth}] (A-3)--(A-2);
\draw[->, arrows={-Stealth}] (A-2)--(A-4);
\draw[->, arrows={-Stealth}] (A-4)--(B-2);
\draw[->, arrows={-Stealth}] (B-2)--(A-5);
\draw[->, arrows={-Stealth}] (A-5)--(B-3);
\draw[->, arrows={-Stealth}] (B-3)--(A-6);
\draw[->, arrows={-Stealth}] (A-6)--(A-4);
\draw[->, arrows={-Stealth}] (A-4)--(B-1);
\end{tikzpicture}

%% file: graph_ex2.tex
\begin{tikzpicture}[scale=1.1]
\coordinate (A1) at (1,6);
\coordinate (A2) at (4,6);
\coordinate (A3) at (6,7);
\coordinate (A4) at (3,4);
\coordinate (A5) at (5,1);
\coordinate (A6) at (2,2);
\coordinate (B1) at (2,5);
\coordinate (B2) at (6,4);
\coordinate (B3) at (3,1);
\coordinate (B4) at (1,1);
\coordinate (B5) at (5,3);

\node[circle,draw,inner sep=0.5pt](A-1) at (A1) {\tiny $a_1$};
\node[circle,draw,inner sep=0.5pt](A-2) at (A2) {\tiny $a_2$};
\node[circle,draw,inner sep=0.5pt](A-3) at (A3) {\tiny $a_3$};
\node[circle,draw,inner sep=0.5pt](A-4) at (A4) {\tiny $a_4$};
\node[circle,draw,inner sep=0.5pt](A-5) at (A5) {\tiny $a_5$};
\node[circle,draw,inner sep=0.5pt](A-6) at (A6) {\tiny $a_6$};
\node[rectangle,draw, fill=gray!10, inner sep=1.5pt](B-1) at (B1) {\tiny $b_1$};
\node[rectangle,draw, fill=gray!10, inner sep=1.5pt](B-2) at (B2) {\tiny $b_2$};
\node[rectangle,draw, fill=gray!10, inner sep=1.5pt](B-3) at (B3) {\tiny $b_3$};
\node[rectangle,draw, inner sep=1pt](B-4) at (B4) {\tiny $b_4$};
\node[rectangle,draw, inner sep=1pt](B-5) at (B5) {\tiny $b_5$};

\draw[->, arrows={-Stealth}] (A-1)--(B-1);
\draw[->, arrows={-Stealth}] (A-2)--(A-1);
\draw[->, arrows={-Stealth}] (B-2)--(A-2);
\draw[->, arrows={-Stealth}] (A-3) to[out=220,in=110] (B-2);
\draw[->, arrows={-Stealth}] (B-2) -- (A-3);
\draw[->, arrows={-Stealth}] (A-5)--(B-2);
\draw[->, arrows={-Stealth}] (B-3)--(A-5);
\draw[->, arrows={-Stealth}] (B-3)--(A-6);
\draw[->, arrows={-Stealth}] (A-6)--(A-4);
\draw[->, arrows={-Stealth}] (A-4)--(B-3);
\end{tikzpicture}

%% file: graph_ex3.tex
\begin{tikzpicture}[scale=1.1]
\coordinate (A1) at (1,6);
\coordinate (A2) at (4,6);
\coordinate (A3) at (6,7);
\coordinate (A4) at (3,4);
\coordinate (A5) at (5,1);
\coordinate (A6) at (2,2);
\coordinate (B1) at (2,5);
\coordinate (B2) at (6,4);
\coordinate (B3) at (3,1);
\coordinate (B4) at (1,1);
\coordinate (B5) at (5,3);

\node[circle,draw,inner sep=0.5pt](A-1) at (A1) {\tiny $a_1$};
\node[circle,draw,inner sep=0.5pt](A-2) at (A2) {\tiny $a_2$};
\node[circle,draw,inner sep=0.5pt](A-3) at (A3) {\tiny $a_3$};
\node[circle,draw,inner sep=0.5pt](A-4) at (A4) {\tiny $a_4$};
\node[circle,draw,inner sep=0.5pt](A-5) at (A5) {\tiny $a_5$};
\node[circle,draw,inner sep=0.5pt](A-6) at (A6) {\tiny $a_6$};
\node[rectangle,draw, fill=gray!10, inner sep=1.5pt](B-1) at (B1) {\tiny $b_1$};
\node[rectangle,draw, fill=gray!10, inner sep=1.5pt](B-2) at (B2) {\tiny $b_2$};
\node[rectangle,draw, fill=gray!10, inner sep=1.5pt](B-3) at (B3) {\tiny $b_3$};
\node[rectangle,draw, inner sep=1pt](B-4) at (B4) {\tiny $b_4$};
\node[rectangle,draw, inner sep=1pt](B-5) at (B5) {\tiny $b_5$};

\draw[->, arrows={-Stealth}] (B-1)--(A-1);
\draw[->, arrows={-Stealth}] (A-1)--(A-2);
\draw[->, arrows={-Stealth}] (A-2)--(A-3);
\draw[->, arrows={-Stealth}] (A-3) -- (B-2);
\draw[->, arrows={-Stealth}] (B-2)--(A-4);
\draw[->, arrows={-Stealth}] (A-4)--(A-5);
\draw[->, arrows={-Stealth}] (A-5)--(B-3);
\draw[->, arrows={-Stealth}] (B-3)--(A-6);
\draw[->, arrows={-Stealth}] (A-6)--(A-4);
\draw[->, arrows={-Stealth}] (A-4)--(B-1);
\end{tikzpicture}

%% file: graph_ex4.tex
\begin{tikzpicture}[scale=1.1]
\coordinate (A1) at (1,6);
\coordinate (A2) at (4,6);
\coordinate (A3) at (6,7);
\coordinate (A4) at (3,4);
\coordinate (A5) at (5,1);
\coordinate (A6) at (2,2);
\coordinate (B1) at (2,5);
\coordinate (B2) at (6,4);
\coordinate (B3) at (3,1);
\coordinate (B4) at (1,1);
\coordinate (B5) at (5,3);

\node[circle,draw,inner sep=0.5pt](A-1) at (A1) {\tiny $a_1$};
\node[circle,draw,inner sep=0.5pt](A-2) at (A2) {\tiny $a_2$};
\node[circle,draw,inner sep=0.5pt](A-3) at (A3) {\tiny $a_3$};
\node[circle,draw,inner sep=0.5pt](A-4) at (A4) {\tiny $a_4$};
\node[circle,draw,inner sep=0.5pt](A-5) at (A5) {\tiny $a_5$};
\node[circle,draw,inner sep=0.5pt](A-6) at (A6) {\tiny $a_6$};
\node[rectangle,draw, fill=gray!10, inner sep=1.5pt](B-1) at (B1) {\tiny $b_1$};
\node[rectangle,draw, fill=gray!10, inner sep=1.5pt](B-2) at (B2) {\tiny $b_2$};
\node[rectangle,draw, fill=gray!10, inner sep=1.5pt](B-3) at (B3) {\tiny $b_3$};
\node[rectangle,draw, inner sep=1pt](B-4) at (B4) {\tiny $b_4$};
\node[rectangle,draw, inner sep=1pt](B-5) at (B5) {\tiny $b_5$};

\draw[->, arrows={-Stealth}] (B-1)--(A-1);
\draw[->, arrows={-Stealth}] (A-1)--(A-2);
\draw[->, arrows={-Stealth}] (A-2)--(A-3);
\draw[->, arrows={-Stealth}] (A-3) -- (B-2);
\draw[->, arrows={-Stealth}] (B-2)--(A-4);
\draw[->, arrows={-Stealth}] (A-4)--(A-5);
\draw[->, arrows={-Stealth}] (A-5)--(B-3);
\draw[->, arrows={-Stealth}] (B-3)--(A-6);
\draw[->, arrows={-Stealth}] (A-6)--(B-1);
\end{tikzpicture}